\documentclass[10pt]{amsart}
\usepackage{amsmath,amssymb,amsthm,graphicx,mathrsfs,url}
\usepackage[usenames,dvipsnames]{color}
\definecolor{darkred}{rgb}{0.4,0.1,0.1}
\definecolor{darkred2}{rgb}{0.4,0.1,0.1}
\definecolor{darkblue}{rgb}{0.1,0.1,0.4}
\definecolor{darkgrey}{rgb}{0.5,0.5,0.5}
\usepackage{tikz}
\usetikzlibrary{arrows,shapes}
\usepackage{xifthen}
\tikzset{
>=stealth',
help lines/.style={dashed, thick},
axis/.style={<->},
important line/.style={thick},
connection/.style={thick, dotted},
}
\newcommand{\cw}{0.6}%	Cste + fct

\newcommand{\dom}{{\mathrm{dom\,}}}

 \newtheorem{thm}{Theorem}[section]
 \newtheorem{cor}[thm]{Corollary}
 \newtheorem{lem}[thm]{Lemma}
 \newtheorem{prop}[thm]{Proposition}

 \theoremstyle{remark}
 \newtheorem{rem}[thm]{Remark}
 
 \numberwithin{equation}{section}
\theoremstyle{definition}

\newcommand{\ba}{\begin{array}}
\newcommand{\ea}{\end{array}}
\newcommand{\bea}{\begin{eqnarray}}
\newcommand{\eea}{\end{eqnarray}}
\newcommand{\bead}{\begin{eqnarray*}}
\newcommand{\eead}{\end{eqnarray*}}
\newcommand{\be}{\begin{equation}}
\newcommand{\ee}{\end{equation}}
\newcommand{\bed}{\begin{displaymath}}
\newcommand{\eed}{\end{displaymath}}
\newcommand{\bl}{\begin{lem}}
\newcommand{\el}{\end{lem}}
\newcommand{\bp}{\begin{prop}}
\newcommand{\ep}{\end{prop}}
\newcommand{\bt}{\begin{thm}}
\newcommand{\et}{\end{thm}}

\newcommand{\bc}{\begin{cor}}
\newcommand{\ec}{\end{cor}}

\newcommand{\br}{\begin{rem}}
\newcommand{\er}{\end{rem}}
\newcommand{\bd}{\begin{defn}}
\newcommand{\ed}{\end{defn}}

\DeclareMathOperator\Real{Re}
\DeclareMathOperator\Imag{Im}
\renewcommand\Re{\Real}
\renewcommand\Im{\Imag}

\newcommand\cS{\mathcal S}

\newcommand\fra{\mathfrak a}

\newcommand\ov{\overline}
\newcommand\wt{\widetilde}
\newcommand\wh{\widehat}

\newcommand{\defeq}{\mathrel{\mathop:}=}

\newcommand\sess{\sigma_{\rm ess}}

\newcommand\void[1]{}

\def\sess{\sigma_{\rm ess}}

\def\frs{{\mathfrak s}}

      \def\dC{{\mathbb C}}

   \def\dN{{\mathbb N}}   
      \def\dR{{\mathbb R}}

\def\cS{{\mathcal S}}

\begin{document}

\title[Wigner von-Neumann   perturbations of the Kronig-Penney model]{Spectral analysis of the  half-line Kronig-Penney model with Wigner-von Neumann perturbations}

\author{Vladimir Lotoreichik}
\address{
Department of Computational Mathematics,
Graz University of Technology,
Steyrergasse 30, 8010, Graz, Austria
}
\email{lotoreichik@math.tugraz.at}
%\address{Department of Mathematics,
%Saint-Petersburg State University of  Information Technologies, Mechanics and Optics,
%Kronverkskiy, 49, Saint-Petersburg, Russia}

\author{Sergey Simonov}
\address{
School of Mathematical Sciences,
Dublin Institute of Technology,
Kevin Street,
Dublin 8,
Ireland}
\email{sergey.a.simonov@gmail.com}



\begin{abstract}
The spectrum of the self-adjoint Schr\"odinger operator associated with the Kronig-Penney model
on the half-line has a band-gap structure:
its absolutely continuous spectrum consists of intervals (bands) separated by gaps.
We show that if one changes strengths of interactions or locations of interaction centers
by adding an oscillating and slowly decaying sequence which resembles the classical
Wigner-von Neumann potential, then this structure of the absolutely continuous
spectrum is preserved. At the same time in each spectral band precisely two critical points appear.
At these points ``instable'' embedded eigenvalues may exist. We obtain locations of the critical points and
discuss for each of them the possibility of an embedded eigenvalue to appear.
We also show that the spectrum in gaps remains discrete.
\end{abstract}

\maketitle

%*************************************************************
\section{Introduction}
%*************************************************************
In the classical paper~\cite{vNW29} von~Neumann and Wigner studied the
one-dimensional Schr\"odinger operator with the potential of the form $\frac{c\sin(2\omega x)}x$
and discovered that such an operator may have an eigenvalue at the point of the continuous
spectrum $\lambda=\omega^2$. Since then such potentials permanently attracted interest
\cite{Al72,Be91,Be94,HKS91,L10,Ma73,N07, RS78}.
See also recent developments in this context
on the counterpart problem for Jacobi matrices
\cite{JS10, S12} and for the case of periodic differential Schr\"odinger operators with Wigner-von Neumann perturbations \cite{KN07,KS11,LO13,NS11}.
{
In parallel to the progress in the investigation of
 Wigner-von Neumann potentials considerable
interest was attracted by Schr\"odinger
operators with point $\delta$-interactions
\cite{GK85, GO10, Ko89, KM10, L11, M95, SCS94}
and also with more general distributional potentials \cite{EGNT13,ET13,R05, SS99,SS03} (our list of references is of course by no means complete).

So let the discrete set $X=\{x_n\colon n\in\mathbb N\}\subset\dR_+$
with elements enumerated in the increasing order
be such that
\[
0 < \inf_{n\in\dN}|x_{n+1} -x_n| \le
\sup_{n\in\dN}|x_{n+1} -x_n| <\infty
\]
and the real-valued sequence $\alpha=\{\alpha_n\}_{n\in\mathbb N}$ be arbitrary.
We deal with the self-adjoint Schr\"odinger operator
$H_{\varkappa,X,\alpha}$ in $L^2(\mathbb R_+)$
with point interactions of strengths $\alpha$ supported on $X$.
This operator corresponds
to the formal differential expression
\begin{equation*}
\label{expression}
-\frac{d^2}{dx^2} + \sum_{n\in\dN}\alpha_n\delta_{x_n},
\end{equation*}
and to the boundary condition at the origin
\begin{equation}\label{bc}
    \psi(0)\cos\varkappa-\psi'(0)\sin\varkappa = 0,
    \qquad\varkappa\in[0,\pi).
\end{equation}
See Section~\ref{sec:def}
for the mathematically rigorous definition of such operators.
As a special case
the Kronig-Penney model corresponds to the self-adjoint operator
$H_\varkappa: = H_{\varkappa,d\dN,\{\alpha_0\}}$ as above
with some $d > 0$ and $\alpha_0\in\dR$.
It describes the behaviour of a free non-relativistic charged quantum particle interacting with
the lattice $d\dN$.
The constant $\alpha_0$ characterizes the strength of
interaction between the
particle and each interaction center in the lattice.
This interaction can be repulsive ($\alpha_0 >0$), attractive
($\alpha_0 <0$) or absent ($\alpha_0 =0$).
 The spectrum of the operator $H_{\varkappa}$
has a band-gap structure: it consists of infinitely many bands of the purely
absolutely continuous spectrum and outside these bands the spectrum of  $H_{\varkappa}$ is discrete, cf.
\cite[Chapter III.2]{AGHH05}.
The operator  $H_\varkappa$ was first investigated in the classical paper~\cite{KP31} by
Kronig and Penney.

In the present paper we study what happens with
the spectrum of  the Kronig-Penney model  in the case
of perturbation of strengths or positions of interactions
by a slowly decaying oscillating sequence resembling the Wigner-von Neumann potential.
Let  the
constants $d,\alpha_0,c,\omega,\phi,\gamma$ and a
real-valued sequence $\{q_n\}_{n\in\mathbb N}$ be such that
\begin{equation}\label{conditions on constants}
\begin{split}
& d > 0,~ \alpha_0 \in\dR,~c\in\dR\setminus\{0\},     ~\omega\in{ \left(0,\pi\right)\setminus\{\pi/2\}},\\
    &
    { \phi \in[0,2\pi)},~\gamma\in\left(\frac12,1\right],
    ~\{q_n\}_{n\in\dN}\in \ell^1(\dN).
    \end{split}
\end{equation}

\emph{Model I: Wigner-von Neumann amplitude perturbation.}
We add a discrete Wigner-von Neumann potential to the constant sequence of interaction strengths. Namely, we consider the discrete set $\wt X$ and the sequence of interaction strengths $\wt \alpha$ given by
\begin{equation}\label{wt X}
    \wt x_n := nd, \quad n\in\dN,
\end{equation}
and
\begin{equation}\label{wt alpha}
    \wt\alpha_n:=\alpha_0+\frac{c\sin(2\omega n {+ \phi})}{n^{\gamma}}+q_n,\quad n\in\dN.
\end{equation}
We study the self-adjoint operator $H_{\varkappa,\wt X,\wt \alpha}$
which reflects an amplitude  perturbation of the Kronig-Penney model.

\emph{Model II: Wigner-von Neumann positional perturbation.}
We change the distances between interaction centers in a ``Wigner-von Neumann''  way, i.e.\textcolor{red}{,}
we add a sequence of the form of Wigner-von Neumann potential to the coordinates
of interaction centers leaving the strengths constant.
Let the discrete set $\wh X$ and the sequence of strengths $\wh\alpha$ be \begin{equation}\label{wh X}
    \wh x_n :=nd+\frac{c\sin(2\omega n { +\phi})}{n^{\gamma}}+q_n, \quad n\in\dN,
\end{equation}
and
\begin{equation}\label{wh alpha}
    \wh\alpha_n := \alpha_0, \quad n\in\dN.
\end{equation}
We study the operator $H_{\varkappa,\wh X,\wh \alpha}$
which reflects a positional perturbation of the Kronig-Penney model
and describes properties of one-dimensional crystals with global defects.

%
%\centerline{\bf --------Suggestion--------------}
%%--------------------------
%
%\emph{Model 3: Slowly decaying
%complex perturbations.}
%Let the discrete set $X_{3} = \{x_n^{(3)}\colon n\in\dN\}$ be as
%defined as
%\[
%x_n^{(3)} := nd
%\]
%and the sequence of strengths be defined as
%\begin{equation}\label{wh alpha}
%   \alpha_n^{(3)} := \alpha_0  +\frac{\rm i c}{n}, \quad n\in\dN.
%\end{equation}
%We study the operator $H_{\varkappa}^{(3)}$
%which formally corresponds to the differential expression
%\begin{equation*}
%	-\frac{d^2}{dx^2} + \sum_{n\in\dN}\alpha_n^{(3)}
%	\delta_{x_n^{(3)}}\langle\delta_{x_n^{(3)}},\cdot\rangle
%\end{equation*}
%with the boundary condition \eqref{bc}.
%%--------------------------
%
%\centerline{\bf ---------------------------}

We also mention that local defects in the Kronig-Penney model are discussed in \cite[\S III.2.6]{AGHH05};
situations of random perturbations of positions  were recently considered
in~\cite{HIT10}.

The essential spectrum of the operator $H_\varkappa$ has a band-gap structure
similar to the case of Schr\"odinger operator with regular periodic potential:
\begin{equation*}
\begin{split}
	&\sigma_{\rm ess}(H_\varkappa) =
    \sigma_{\rm ac}(H_\varkappa) =
    \bigcup_{n=1}^\infty
    \Bigl(\bigl[\lambda_{2n-1}^+,	
	\lambda_{2n-1}^-\bigr]\cup \bigl[\lambda_{2n}^-,\lambda_{2n}^+\bigr]\Bigr),\\
	&\text{where }  { \lambda_1^+ < \lambda_1^- \le \lambda_2^- < \lambda_2^+ \le \dots \le \lambda_{2n-1}^+ 	
	<\lambda_{2n-1}^- \le \lambda_{2n}^-<\lambda_{2n}^+  \le \dots}
\end{split}
\end{equation*}
The locations of boundary points of the spectral bands are determined
by the parameters $\alpha_0$ and $d$. Namely, the values $\lambda_n^\pm$ are the $n$-th roots of
the corresponding Kronig-Penney equations
\begin{equation*}
    { L}\left(\sqrt{\lambda}\right) = \pm 1,\quad
    { \lambda\in\dR,~\Im\sqrt{\lambda} \ge 0,}
\end{equation*}
where
\begin{equation}\label{Lyapunov function for delta}
    { L}(k) = \cos(kd)+\alpha_0 \frac{\sin(kd)}{2k},
\end{equation}
and the value $L(0)$ is defined via extension by continuity
{ to the point $k = 0$ of the above function}.
%$\lambda_{2n-1}^- = \bigl(\frac{\pi(2n-1)}{d}\bigr)^2$, $\lambda_{2n}^+ = \bigl(\frac{2\pi n}{d}\bigr)^2$.
For the details the reader is referred to the monograph~
\cite[Chapter III.2]{AGHH05}.

%One has:
%\begin{equation*}
%	0<\mu_1^+<\mu_1^-<\mu_2^-<\mu_2^+<\dots<
%        \mu_{2n-1}^+<\mu_{2n-1}^-<\mu_{2n}^-<\mu_{2n}^+<\dots
%\end{equation*}
%and
%\begin{equation*}
%\begin{split}
%	&\sigma_{\rm ess}(H'_\varkappa) =
%   \sigma_{\rm ac}(H'_\varkappa) =
%    \bigcup_{n=1}^\infty
%   \Bigl(\bigl[\mu_{2n-1}^+,\mu_{2n-1}^-\bigr]
%    \cup
%    \bigl[\mu_{2n}^-,\mu_{2n}^+\bigr]\Bigr).
%\end{split}
%\end{equation*}
%The values  $\mu_n^\pm$ are the $n$-th roots  of the corresponding equations
%\begin{equation}
%\label{eq:delta'}
%   L_{\delta'}\left(\sqrt{\mu}\right)=\pm 1,
%\end{equation}
%where
%\begin{equation}\label{Lyapunov function for delta'}
%    L_{\delta'}(k)=\cos(kd)-\alpha_0\frac{k\sin(kd)}2,
%\end{equation}
%and, in particular,
%$\mu_{2n-1}^- = \bigl(\frac{\pi(2n-1)}{d}\bigr)^2$ and
%$\mu_{2n}^+ = \bigl(\frac{2\pi n}{d}\bigr)^2$.

In the present paper we show that the absolutely continuous spectra of the operators $H_{\varkappa,\wt X,\wt \alpha}$ and
$H_{\varkappa,\wh X,\wh \alpha}$ coincide with the absolutely continuous spectrum of the non-perturbed operator $H_{\varkappa}$.  However the spectrum in bands may not remain purely absolutely continuous. Namely,
at certain points which are called critical embedded eigenvalues may appear.
In each band there are two such points.
The critical points $\lambda_{n,\rm cr}^\pm$ in the $n$-th spectral band
are the  $n$-th roots of the equations
\begin{equation*}
   { L}\left(\sqrt{\lambda}\right) = \pm\cos\omega, \quad
    { \lambda\in\dR,~\Im\sqrt{\lambda} \ge 0.}
\end{equation*}
The illustration is given in Figure \ref{fig1}.
 
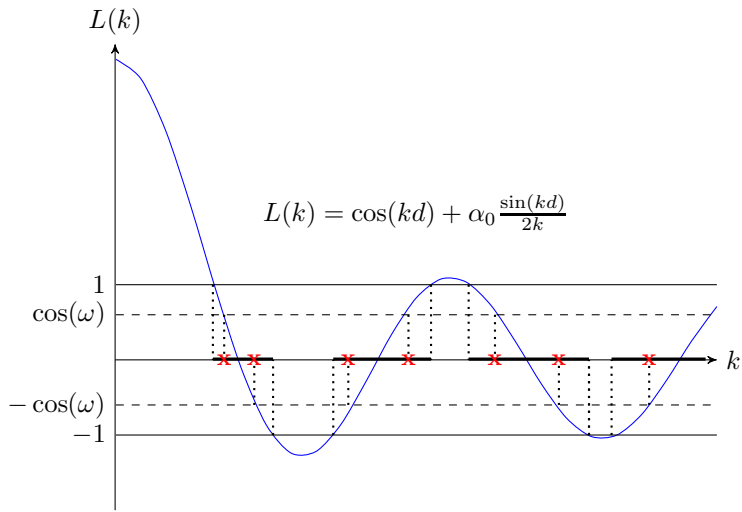
\begin{figure}[h]
\label{fig1}
\begin{tikzpicture}[domain=0:7.8]
    \draw[->] (0.0,0) -- (8.0,0) node[right] {$k$};
    \draw[->] (0,-2.0) -- (0,4.2) node[above] {$L(k)$};
\draw [blue, domain= 0.01:8, smooth] plot (\x, {(4.0 *  sin((\x * 1.5) r ) / (2.0 * \x) + cos((\x * 1.5)  r))});
\draw (4.0, 1.6)  node[above] {$L(k) = \cos(kd)+  \alpha_0 \frac{\sin(kd)}{2k}$};
 \draw (0.0,-1) node[left] {$-1$};
 \draw  (0.0,1) node[left] {$1$};
 \draw (0.0,0.6) node[left] {$\cos(\omega)$};
 \draw  (0.0,-0.6) node[left] {$-\cos(\omega)$};

\draw [black, domain= 0.01:8, smooth, dashed] plot (\x, {0.6});
\draw [black, domain= 0.01:8, smooth, dashed] plot (\x, {-0.6});
\draw [black, domain= 0.01:8, smooth] plot (\x, {1.0});
\draw [black, domain= 0.01:8, smooth] plot (\x, {-1.0});

%band 1   
\draw[connection] (1.3,1.00) -- (1.3,0.0) node[below] {
%{\small $\begin{smallmatrix}\lambda_1^+ \lambda_{1,\mathrm cr}^+\end{smallmatrix}$}
};

\draw[connection] (2.1,-1.00) -- (2.1,0.0) node[above] {
%{\small $\begin{smallmatrix}\lambda_1^-\end{smallmatrix}$}
};
\draw[very thick, color=black] (1.3,0.01) -- (2.1, 0.01);

\draw[connection] (1.45,\cw) -- (1.45,0); 
\draw[color = red, very thick] (1.45,0.0) node[] {\small \bf x};

\draw[connection] (1.85,-\cw) -- (1.85,0);
\draw[color = red, very thick] (1.85,0.0) node[] {\small \bf x};

%band 2   
\draw[connection] (2.9,-1.00) -- (2.9,0.0) node[below] {
%{\small $\begin{smallmatrix}\lambda_1^+ \lambda_{1,\mathrm cr}^+\end{smallmatrix}$}
};

\draw[connection] (4.2,1.00) -- (4.2,0.0) ;
\draw[very thick, color=black] (2.9,0.01) -- (4.2, 0.01);

\draw[connection] (3.1,-\cw) -- (3.1,0); 
%$\begin{smallmatrix}\lambda^+_{1,\rm cr}\end{smallmatrix}$};
\draw[color = red, very thick] (3.1,0.0) node[] {\small \bf x};

\draw[connection] (3.9,\cw) -- (3.9,0);
\draw[color = red, very thick] (3.9,0) node[] {\small \bf x};

%band 3 
\draw[connection] (4.7,1.00) -- (4.7,0.0);

\draw[connection] (6.3,-1.00) -- (6.3,0.0);
\draw[very thick, color=black] (4.7,0.01) -- (6.3, 0.01);

\draw[connection] (5.9,-\cw) -- (5.9, 0); 
\draw[color = red, very thick] (5.9,0) node[] {\small \bf x};

\draw[connection] (5.05,\cw) -- (5.05,0);
\draw[color = red, very thick] (5.05,0) node[] {\small \bf x};

%band 4 
\draw[connection] (6.6,-1.00) -- (6.6,0.0);;

\draw[connection] (6.3,-1.00) -- (6.3,0.0);
\draw[very thick, color=black] (6.6,0.01) -- (7.85, 0.01);

\draw[connection] (7.1,-0.6) -- (7.1,  0); 
\draw[color = red, very thick] (7.1,0) node[] {\small \bf x};
\end{tikzpicture}
\caption{{\em The curve} is the graph of $L(\cdot)$; {\em bold intervals} are bands of 
the absolutely continuous spectrum; {\em crosses} denote the critical points $\{\lambda_{n,\rm cr}^\pm\}_{n\in\dN}$.}
\end{figure}

For the considered operators  we give exact conditions which ensure that a given critical point  is
indeed  an embedded eigenvalue for some $\varkappa\in[0,\pi)$.  For a given point this can occur
only for one value of $\varkappa\in[0,\pi)$. We calculate the asymptotics of  generalized
eigenfunctions for all  values of the spectral parameter $\lambda \in\dR$, except the endpoints of the bands. The possibility of the appearance of an embedded
eigenvalue at certain critical point depends on the rate of the decay of the subordinate generalized
eigenfunction.  It turns out that for Model I due to { the} decay rate of generalized eigenfunctions embedded eigenvalues can appear only in the low lying bands of the absolutely continuous spectrum, whereas for Model II one can create arbitrarily large embedded eigenvalues varying $\varkappa\in[0,\pi)$. We also show that the spectrum in gaps remains discrete.

%In the case of $\delta'$-type potentials the critical points in the $n$-th spectral band $\mu_{n,\rm cr}^\pm$ are the $n$-th %roots of the equations
%\begin{equation*}
%   L_{\delta'}\left(\sqrt\mu\right)=\pm\cos\omega,
%\end{equation*}
%The illustration is given on Figure~2.
%\input{Fig2}

Our results are close to the results for one-dimensional Schr\"od\-inger operator
with the Wigner-von Neumann potential and a periodic background potential. Such  operators
were considered recently in~\cite{KN07, KS11, LO13, NS11}.
%Here we obtain asymptotics of the subordinate solutions explicitly in terms of the parameters.
%In this sense our model has a remarkable property to be exactly  solvable.

To study spectra in bands we make a discretization
of the spectral equations and further we perform an asymptotic integration of { the} obtained discrete linear system
using Benzaid-Lutz-type theorems \cite{BL87}. As the next step we apply a modification of
Gilbert-Pearson subordinacy theory \cite{GP87}.
To study spectra in gaps we use compact
perturbation argument.

The reader can trace some analogies of our case with Jacobi matrices. The coefficient matrix of the
discrete linear system that appears in our analysis has a form similar to the transfer-matrix
for some Jacobi matrix.

The body of the paper contains two parts: the preliminary part --- which consists of mostly known
material --- and the main part, where we obtain  new results.
In the preliminary part we give a
rigorous definition of one-dimensional Schr\"odinger operators with $\delta$-interactions (Section~\ref{sec:def}),
show how to reduce the spectral equations for these operators
to  discrete linear systems in $\dC^2$ (Section~\ref{sec:system}), provide a formulation of { the analogue of the  subordinacy theory} for { the considered operators} (Section~\ref{sec:subord}). Further we
formulate few results from asymptotic integration theory for discrete linear
systems (Section~\ref{sec:BL}).
In the main part, in Section~\ref{sec:class} we study a special class of discrete linear systems in $\dC^2$ and find asymptotics of solutions of these systems. After certain technical preliminary calculations in Section~\ref{sec:decomp},
we proceed to Section~\ref{sec:main}, where we obtain asymptotics of generalized { eigenfunctions}
for Schr\"odinger operators with point interactions subject to Model I and Model II. Further we pass to the conclusions about  the spectra in bands putting an emphasis on critical points.
In Section~\ref{sec:compact} we prove compactness
of resolvent differences of two Schr\"odinger operators with
point interactions under certain assumptions on interaction
strengths and positions of interactions. This result
is used to show that the  spectrum  in gaps of the Schr\"odinger
operators with point interactions subject to Models I and II
remains discrete.

%----------------------------------------------------------------
\subsection*{Notations}
%----------------------------------------------------------------
By small letters with integer subindices, e.g. $\xi_n$, we denote sequences of
complex numbers. By small letters with integer subindices and arrows above, e.g. $\overrightarrow u_n$, we denote sequences of  $\dC^2$-vectors.
By capital letters with integer indices, e.g. $R_n$, we denote sequences of
$2\times 2$ matrices with complex entries. We use notations
$\ell^p(\dN)$, $\ell^p(\dN,\dC^2)$ and $\ell^p(\dN,\dC^{2\times 2})$  for spaces of  summable ($p=1$), square-summable ($p=2$) and bounded ($p=\infty$) sequences of  complex values, complex two-dimensional vectors and complex $2\times 2$ matrices, respectively.
{ For a self-adjoint operator $H$} we denote  its pure point, absolutely continuous, singular continuous, essential and discrete spectra by
$\sigma_{\rm pp}(H)$, $\sigma_{\rm ac}(H)$, $\sigma_{\rm sc}(H)$,  $ \sigma_{\rm ess}(H)$ and
$\sigma_{\rm d}(H)$, respectively. { Throughout the paper
for $\lambda\in\dR$ we choose by default the branch of the square root so that $\Im\sqrt{\lambda} \ge 0$ if the opposite is not said.}

%*****************************************************************
\section{Preliminaries}
\label{sec:prelim}
%*****************************************************************
%In Section~\ref{sec:def} we define Schr\"odinger operators with point interactions rigorously,
%clarify what is spectral equations for such operators and formulate subordinacy principle
%for the spectral analysis of such operators in terms of asymptotical properties of solutions
%for spectral equations. Section~\ref{sec:system} contains a reduction to discrete linear systems
% of spectral equations introduced in Section~\ref{sec:def}. In Section~\ref{sec:BL} we formulate
%certain simplifications of Benzaid-Lutz theorems suitable for our purposes.
%Finally, in Section~\ref{sec:Weyl} we give another approach to the spectral analysis
%of operators with point interactions based on Weyl functions associated with ordinary boundary triples.

%-----------------------------------------------------------------
\subsection{Definition of operators with point interactions}
%-----------------------------------------------------------------
\label{sec:def}
In this section we give a rigorous definition of operators with $\delta$-interactions, see, e.g.,
\cite{GK85, Ko89}.
Let $\alpha=\{\alpha_n\}_{n\in\dN}$ be a sequence of real numbers and let $X = \{x_n\colon n\in\dN\}$ be a discrete set on $\dR_+$ ordered as $0<x_1<x_2< \dots$. Assume that the set $X$ satisfies
\begin{equation}\label{condition on X}
	\inf_{n\in\dN}\big|x_{n+1}-x_n\big|>0\quad\text{and}\quad \sup_{n\in\dN}\big|x_{n+1}-x_n\big|<\infty.
\end{equation}
Denote also $x_0:=0$. In order to define the operator corresponding to the formal expression\footnote{In \eqref{diffexpression} we denote usual $\delta$-distribution supported on $x_n$ by $\delta_{x_n}$.}
\begin{equation}
\label{diffexpression}
    \tau_{X,\alpha}:=-\frac{d^2}{dx^2}+\sum_{n\in\mathbb N}\alpha_n\delta_{x_n}
\end{equation}
and the boundary condition
\begin{equation*}
    \psi(0)\cos\varkappa-\psi'(0)\sin\varkappa = 0,
\end{equation*}
consider the following set of functions:
\begin{equation*}
	\cS_{X,\alpha} \defeq
	\Biggl\{
 	\psi:\ \psi,\psi^{\prime} \in AC_{\rm loc}(\dR_+\setminus X),
	\begin{smallmatrix}
	 	\psi(x_n+) = \psi(x_n-)=\psi(x_n)
        \\
		\psi^\prime(x_n+) - \psi^{\prime}(x_n-) = \alpha_n\psi(x_n)
	\end{smallmatrix},
    ~n\in\dN\Biggr\}
\end{equation*}
and let the operator $H_{\varkappa,X,\alpha}$ be defined in $L^2(\mathbb R_+)$ by its action
\begin{equation*}
\label{action}
H_{\varkappa, X, \alpha}\psi:= -\psi''
\end{equation*}
on the domain
\begin{equation*}
	\dom H_{\varkappa, X, \alpha}
                    :=
	\left\{
 	\psi \in\cS_{X,\alpha}\colon \psi,
 	 \psi''\in L^2(\dR_+),\psi(0)\cos\varkappa = \psi^\prime(0)\sin\varkappa
	\right\}.
\end{equation*}
According to \cite[Theorem 3.1]{GK85} the operator $H_{\varkappa,X,\alpha}$
is self-adjoint.

The  spectral equation $\tau_{X,\alpha}\psi=\lambda\psi$ is understood as the equation $-\psi''(x)=\lambda\psi(x)$ for $\psi\in\cS_{X,\alpha}$. The latter is equivalent to the following system:
\begin{equation}\label{spectral equation for delta}
\begin{split}
    &-\psi''(x) =\lambda\psi(x),\quad x\in\dR_+\setminus X,
    \\[0.5ex]
    &\psi(x_n+)=\psi(x_n-),\ \psi'(x_n+)=\psi'(x_n-)+\alpha_n\psi(x_n-),\quad n\in\dN.
\end{split}
\end{equation}
The equation \eqref{spectral equation for delta} has two linearly independent
solutions which are called generalized eigenfunctions.
If $\psi\in L^2(\dR_+)$ satisfies~\eqref{spectral equation for delta} and the boundary condition at the origin,
then $\psi$ is an eigenfunction of $H_{\varkappa,X,\alpha}$.

%One can define the operator $H'_{\varkappa,X,\alpha}$ corresponding to the expression
%\begin{equation*}
%	\tau'_{X,\alpha}:=-\frac{d^2}{dx^2}+\sum_{n\in\mathbb N}\alpha_n\delta'_{x_n}\langle\delta_{x_n}',\cdot\rangle
%\end{equation*}
%with the boundary condition
%\begin{equation*}
%    \psi(0)\cos\varkappa-\psi'(0)\sin\varkappa = 0,
%\end{equation*}
%in an analogous way, if the set $X = \{x_n\}_{n\in\dN}$ and the sequence $\alpha = \{\alpha_n\}_{n\in\dN}$ are given and the condition \eqref{condition on X} holds. Let
%\begin{equation*}
%    \cS'_{X,\alpha} \defeq
%	\Biggl\{
% 	\psi:\ \psi,\psi^{\prime} \in AC_{\rm loc}(\dR_+\setminus X),
%	\begin{smallmatrix}
%		\psi^\prime(x_n+) = \psi^\prime(x_n-)=\psi^\prime(x_n)\\
%		 \psi(x_n+) - \psi(x_n-) = \alpha_n\psi^\prime(x_n)
%	\end{smallmatrix},~n\in\dN
%	\Biggr\}
%\end{equation*}
%and define
%\begin{equation*}
%    H'_{\varkappa, X,\alpha}\psi:=-\psi^{\prime\prime}
%\end{equation*}
%on the domain
%\begin{equation*}
%	\dom H'_{\varkappa, X,\alpha} :=
%	\left\{
% 	\psi \in L^2(\dR_+)\cap\cS'_{X,\alpha}\colon\psi''\in L^2(\dR_+),
%    \psi(0)\cos\varkappa = \psi^\prime(0)\sin\varkappa
%	\right\}.
%\end{equation*}
%The spectral equation for the differential expression $\tau'_{X,\alpha}$ has the form
%\begin{equation}\label{spectral equation for delta'}
%\begin{split}
%&-\psi''(x)=\lambda\psi(x),\quad x\in\dR_+\setminus X, \\[0.5ex]
%&\psi(x_n+)=\psi(x_n-)+\alpha_n\psi'(x_n-),\ \psi'(x_n+)=\psi'(x_n-), \quad n\in\dN.
%    \end{split}
%\end{equation}

%----------------------------------------------------------------
\subsection{Reduction of the eigenfunction equation to a discrete linear system}
%----------------------------------------------------------------
\label{sec:system}
In this subsection we recall rather well-known way of reduction of the spectral equation
\eqref{spectral equation for delta} to a discrete linear system.
Let the discrete set $X$ and the sequence of strengths $\alpha$ be as in the previous section. { Fix $\lambda\in\dR\setminus\{0\}$
and set $k := \sqrt{\lambda}$.} To make our formulas more compact we introduce the following notations
\begin{equation*}
    s_n(k):=\sin\bigl(k(x_{n+1}-x_n)\bigr)\quad\text{and}\quad c_n(k):=\cos\bigl(k(x_{n+1}-x_n)\bigr), \qquad n\in\dN_0.
\end{equation*}

\begin{rem}
 Note that for $\lambda < 0$ the value $k$ is purely imaginary, and in this case we use identities $\sin(i\alpha) = i \sinh(\alpha)$ and $\cos(i\alpha) =\cosh(\alpha)$.
\end{rem}

{ For a solution $\psi$ of \eqref{spectral equation for delta} we introduce the sequence
\begin{equation*}
%\label{xi}
    \xi_n:=\psi(x_n),\quad n\in\dN_0.
\end{equation*}
Assume that the condition  \eqref{condition on X} is satisfied
and $s_n(k) \ne 0$ for all $n \ge N(k)$ with some
$N(k)\in\dN$. Then by
\cite[Chapter III.2]{AGHH05}  for $n\ge N(k)+1$ one has that
\begin{equation}\label{equation for xi}
	-k\Biggl(\frac{\xi_{n-1}}{s_{n-1}(k)}+\frac{\xi_{n+1}}{s_n(k)}\Biggr)
    +\Biggl(\alpha_n
    +k\Biggl(\frac{c_{n-1}(k)}{s_{n-1}(k)}+ \frac{c_n(k)}{s_n(k)}\Biggr)\Biggr)\xi_n=0.
\end{equation}
Inversely, solutions of the eigenfunction equation on each of the intervals $(x_n,x_{n+1})$ can be recovered from their values at the endpoints $x_n$ and $x_{n+1}$:
\begin{equation}\label{psi from xi}
     \psi(x)=\frac{\xi_n\sin(k(x_{n+1}-x))+\xi_{n+1}\sin(k(x-x_n))}{\sin(k(x_{n+1}-x_n))},
    \quad
x\in[x_n,x_{n+1}].
\end{equation}
%Analogously, for a function $\psi\in\cS'_{X,\alpha}$ we introduce the following sequence
%\begin{equation}\label{eta}
%    \eta_n:= \psi'(x_n),\quad n\in\dN_0.
%\end{equation}
%By \cite[Chapter III, Theorem 3.3]{AGHH05} for $k\in\mathbb C_+\cup \mathbb R$ such that for every $n\in\mathbb N_0$ $s_n(k)\neq 0$ it holds that
%\begin{equation}\label{equation for eta}
%	-\frac{1}{k}\Biggl(\frac{\eta_{n-1}}{s_{n-1}(k)}+\frac{\eta_{n+1}}{s_n(k)}\Biggr)
%	+\Biggl(-\alpha_n
%    +\frac{1}{k}\Bigr(\frac{c_{n-1}(k)}{s_{n-1}(k)}+\frac{c_{n}(k)}{s_{n}(k)}\Bigr)\Biggr)\eta_n=0,\ n\in\mathbb N.
%\end{equation}
%Solutions of the eigenfunction equation on each of the intervals $(x_n,x_{n+1})$ can be recovered from the values of their derivatives at the endpoints $x_n$ and $x_{n+1}$:
%\begin{equation}\label{psi from eta}
%    \psi(x)=\frac{\eta_n\cos(k(x_{n+1}-x))-\eta_{n+1}\cos(k(x-x_n))}{k\sin(k(x_{n+1}-x_n))},
%\quad x\in[x_n,x_{n+1}],\ n\in\mathbb N.
%\end{equation}
The reader may confer with \cite{E97}, where a more general case of a quantum graph is considered.

Instead of working with recurrence relation \eqref{equation for xi}
 we will consider a discrete linear system in $\dC^2$. Define
\begin{equation*}\label{vectors u and v}
    \overrightarrow u_n := \begin{pmatrix}\xi_{n-1} \\ \xi_n \end{pmatrix}.
%\quad\text{and}\quad
%    \overrightarrow v_n := \begin{pmatrix}\eta_{n-1} \\ \eta_n \end{pmatrix}.
\end{equation*}
Observe that \eqref{equation for xi} can be rewritten as
\[
\begin{split}
\xi_{n+1} & = \tfrac{s_n(k)}{k}
\Big(\alpha_n + k\Big(\tfrac{c_{n-1}(k)}{s_{n-1}(k)}
+\tfrac{c_{n}(k)}{s_{n}(k)}\Big)\Big)\xi_n
-\tfrac{s_n(k)}{s_{n-1}(k)}\xi_{n-1}\\
& = \Big(\tfrac{\alpha_n s_n(k)}{k}
 + \tfrac{c_{n-1}(k)s_n(k) + s_{n-1}(k)c_n(k)}{s_{n-1}(k)}\Big)\xi_n
-\tfrac{s_n(k)}{s_{n-1}(k)}\xi_{n-1}\\
& =
\Big(\tfrac{\alpha_n s_n(k)}{k}
 +\tfrac{\sin(k(x_{n+1}-x_{n-1}))}{s_{n-1}(k)}\Big)\xi_n
-\tfrac{s_n(k)}{s_{n-1}(k)}\xi_{n-1}.
\end{split}
\]
The above recurrence relation is then equivalent to
\begin{equation}\label{system for delta}
    \overrightarrow u_{n+1}=T_n(k)\overrightarrow u_n
\end{equation}
with
\begin{equation}\label{T-n-k}
    T_n(k):=\left(
    \begin{array}{cc}
    0 & 1 \\
    -\frac{s_n(k)}{s_{n-1}(k)}
	&
    \frac{\sin(k(x_{n+1}-x_{n-1}))}{s_{n-1}(k)}+\frac{\alpha_ns_n(k)}{k}   \\
    \end{array}
    \right).
\end{equation}
%and \eqref{equation for eta} is equivalent to
%\begin{equation}\label{system for delta'}
%    \overrightarrow v_{n+1}
%    =
%    \left(
%    \begin{array}{cc}
%    0 & 1 \\
%    -\frac{s_n(k)}{s_{n-1}(k)}
%	&
%    \frac{\sin(k(x_{n+1}-x_{n-1})}{s_{n-1}(k)}-\alpha_nks_n(k)   \\
%    \end{array}
%    \right)
%    \vec v_n.
%\end{equation}
The coefficient matrix of this system $T_n(k)$ is called the transfer-matrix.
% and $M_n'(k)$, respectively.
{
\begin{rem}
\label{rem:T-n-0}
The case $\lambda = 0$ requires separate consideration. In this
special case $s_n(k)=0$, but, in all the formulas, expressions of the form $\sin(xk)/k$ should be substituted by $x$, its limit as $k\rightarrow0$ and at the same time solution of the equation $-\psi''(x)=0$ which is zero at { the point $x = 0$} and has the derivative equal to one there, { cf.} \cite{E97}. This gives
\begin{equation}\label{equation for xi 0}
	-\Biggl(\frac{\xi_{n-1}}{x_n-x_{n-1}}+\frac{\xi_{n+1}}{x_{n+1}-x_n}\Biggr)
    +\Biggl(\alpha_n
    +\Biggl(\frac1{x_n-x_{n-1}}+ \frac1{x_{n+1}-x_n}\Biggr)\Biggr)\xi_n=0
\end{equation}
instead of \eqref{equation for xi}. { Equation \eqref{psi from xi}
should be replaced by}
\begin{equation}\label{psi from xi 0}
     \psi(x)=\frac{\xi_n(x_{n+1}-x)+\xi_{n+1}(x-x_n)}{(x_{n+1}-x_n)},\quad x\in[x_n,x_{n+1}]
\end{equation}
{ Instead of \eqref{T-n-k} one gets}
\begin{equation}\label{T-n-k 0}
T_n(0) :=\left(
    \begin{array}{cc}
    0 & 1 \\
    -\frac{x_{n+1}-x_{n}}{x_n -x_{n-1}}
	&
    \frac{x_{n+1}-x_{n-1}}{x_n -x_{n-1}}+\alpha_n(x_{n+1} -x_n)  \\
    \end{array}
    \right).
\end{equation}
\end{rem}
}
%----------------------------------------------------------------
\subsection{Subordinacy}
%----------------------------------------------------------------

 \label{sec:subord}
The subordinacy theory as suggested in \cite{GP87} by D.~Gilbert and D.~Pearson
produced a strong influence on the spectral theory of one-dimensional Schr\"odinger operators. Later on the subordinacy theory was translated to difference equations \cite{KP92}.
For Schr\"odinger operators with $\delta$-interactions there exists a modification of the subordinacy theory,
see, e.g., \cite{SCS94} which  relates the spectral properties of the operator $H_{\varkappa,X,\alpha}$
with the asymptotic behavior of the solutions of the spectral equation~\eqref{spectral equation for delta}.
%A similar type relation holds for the operator $H_{\varkappa, X,\alpha}'$ and the spectral equation \eqref{spectral equation for delta'}.
Analogously to the classical definition of the subordinacy \cite{GP87}
we say that a solution $\psi_1$ of the equation $\tau_{X,\alpha}\psi =\lambda\psi$
%with $\tau = \tau_{X,\alpha}$ or $\tau = \tau_{X,\alpha}'$
is subordinate if and only  if for any other solution $\psi_2$ of the same equation not proportional to
$\psi_1$  the following limiting property holds:
\begin{equation*}
    \lim_{x\rightarrow+\infty}\frac{\int_0^x |\psi_1(t)|^2dt}{\int_0^x |\psi_2(t)|^2dt} = 0.
\end{equation*}
We will use the following propositions to { locate} the absolutely continuous spectrum.

%----------------------------------------------------------------
\begin{prop}\cite[Proposition 7]{SCS94}
\label{prop:stolz}
Let  $H_{\varkappa, X,\alpha}$  be the self-adjoint operator
corresponding to the discrete set $X$ and the sequence of strengths $\alpha$ as in Section~\ref{sec:def}.
Assume that for all $\lambda\subset (a,b)$ there is no subordinate solution for
the spectral equation $\tau_{X,\alpha}\psi = \lambda\psi$.
Then $[a,b] \subset \sigma(H_{\varkappa, X,\alpha})$ and $\sigma(H_{\varkappa, X,\alpha})$
is purely absolutely continuous in $(a,b)$.
%The same is true for the operator $H'_{\varkappa, X,\alpha}$ and the expression $\tau'_{X,\alpha}$.
\end{prop}
%----------------------------------------------------------------

%----------------------------------------------------------------
\begin{prop}
\label{prop:Stolz2}
Let the discrete set $X = \{x_n\colon n\in\dN\}$ on $\dR_+$ be ordered as $0<x_1<x_2< \dots$. Assume that $X$ satisfies conditions
$\Delta x_* := \inf_{n\in\dN} | x_{n+1} - x_n | > 0$ and $\Delta x^* := \sup_{n\in\dN} | x_{n+1} - x_n | < \infty$. Let
$\alpha = \{\alpha_n\}_{n\in\dN}$ be a sequence of real numbers.
Let $\lambda \in \dR$ and set $k :=\sqrt{\lambda}$.
 In the case $\lambda \ne 0$ assume also that $\liminf_{n\rightarrow\infty}
\big| { s_n(k)} \big| > 0$ holds.
If every solution of the equation $\tau_{X,\alpha}\psi = \lambda\psi$ (see \eqref{spectral equation for delta})
is bounded, then for such $\lambda$ there exists no subordinate solution.
\end{prop}
%----------------------------------------------------------------
%{\grey To remove.
%\begin{rem}
%We only need this statement for $\lambda>0$, because using other approach we show in Corollary~\ref{cor:ess} that our perturbed Kronig-Penney
%operators subject to Models I and II have no negative
%absolutely continuous spectra.
%%all the bands of the absolutely continuous spectrum of the unperturbed Kronig-Penney operator defined by the condition $|L_{\delta}(\sqrt{\lambda})|\le1$, see \eqref{Lyapunov function for delta}, lie on the positive half-line for $\alpha_0\ge 0$.
%\end{rem}
%----------------------------------------------------------------
%----------------------------------------------------------------
\begin{proof}
Let $\psi_1$ be an arbitrary solution of $\tau_{X, \alpha}\psi = \lambda\psi$. Set $\xi_n := \psi_1(x_n)$ for $n\in\dN_0$.  If $\lambda \ne 0$, then differentiating \eqref{psi from xi} one gets in the case $s_n(k) \ne 0$
\[
|\psi_1'(x)|\le\frac{|k|(|\xi_n|+|\xi_{n+1}|)}{|s_n(k)|}e^{|k|\Delta x^*}\quad\text{for}~x\in[x_n,x_{n+1}].
\]
There exists $N(k)$ such that $\inf_{n\ge N(k)}|s_n(k)| >0 $. Since the sequence $\{\xi_n\}_{n\in\dN_0}$ is bounded, one has that $\psi_1'(x)$ is also bounded for $x\ge x_{N(k)}$
Obviously, for $x
\le x_{N(k)}$, { $\psi'_1$} is bounded too, since it is piecewise continuous with finite jumps at the points of discontinuity.
If $\lambda = 0$, then one gets
\[
|\psi_1'(x)|\le \frac{|\xi_n| + |\xi_{n+1}|}{\Delta x_*}\quad\text{for}~x\in[x_n,x_{n+1}].
\]
Boundedness of $\{\xi_n\}_{n\in\dN_0}$ and the above formula imply boundedness of $\psi_1'$ in the case $\lambda = 0$ { as well.}

Let $\psi_2$ be any other solution of $\tau_{X, \alpha}\psi = \lambda\psi$, which is linearly independent with $\psi_1$. It follows that there exists a constant $C > 0$ such that
\begin{equation*}
\label{C}
\|\psi_1\|_\infty, \|\psi_2\|_\infty, \|\psi_1'\|_\infty, \|\psi_2'\|_\infty\le C.
\end{equation*}
The Wronskian of the solutions $\psi_1$ and $\psi_2$ is independent of $x$:
\begin{equation}
\label{Wronskian}
W\{\psi_1,\psi_2\}:= \psi_1(x)\psi_2'(x) - \psi_1'(x)\psi_2(x),\quad\text{for all}~ x\in\dR_+\setminus X.
\end{equation}
This is easy to check: it is constant on every interval $(x_n,x_{n+1})$ and at the points $\{x_n\}_{n\in\dN}$ one has
\[
\begin{split}
W\{\psi_1,\psi_2\}(x_n+) &=
\psi_1(x_n+)\psi_2'(x_n+) - \psi_1'(x_n+)\psi_2(x_n+)\\
&= \big(\psi_1(x_n-)\psi_2'(x_n-) + \alpha_n\psi_1(x_n-)\psi_2(x_n-)\big)  \\
&\qquad-\big(\psi_1'(x_n-)\psi_2(x_n-) +
\alpha_n\psi_1(x_n-)\psi_2(x_n-)\big)\\
& = W\{\psi_1,\psi_2\}(x_n-),
\end{split}
\]
where we used \eqref{spectral equation for delta}.
The Wronskian is non-zero since the solutions are linearly independent. From \eqref{Wronskian} one has:
\[
|W\{\psi_1,\psi_2\}|\le C(|\psi_1(x)| + |\psi_1'(x)|),\quad \text{for all}~ x\in\dR_+\setminus X,
\]
and therefore there exist constants $C_*$ and $C^*$ such that
\begin{equation}
\label{Stolz estimate}
0<C_* \le |\psi_1(x)| + |\psi_1'(x)| \le C^*,\quad\text{for all} ~ x\in\dR_+\setminus X.
\end{equation}

Now we apply the trick used in the proof of \cite[Lemma 4]{S92}. We consider for an arbitrary
$n\in\dN$ the interval $[x_n, x_{n+1}]$. Since the function $\psi_1$ is { continuously differentiable} on $[x_n, x_{n+1}]$,
the formula
\begin{equation}
\label{id1}
\psi_1(x_{n+1}) - \psi_1(x_n) = \int_{x_n}^{x_{n+1}}\psi_1'(t)dt
\end{equation}
holds. Set
\[
p_* := \frac{\Delta x_*}{\Delta x_* + 2}.
\]
Next we show that there exists a point
$x_n^*\in[x_n,x_{n+1}]$ such that $|\psi_1(x_n^*)| \ge p_*C_*$.
Let us suppose that such a point does not exist, i.e. $|\psi_1(x)| < p_*C_*$ for all $x\in[x_n,x_{n+1}]$.
We get from \eqref{Stolz estimate} that
$|\psi_1'(x)| > (1 - p_*)C_*$ for every $x\in [x_n,x_{n+1}]$.
In particular $\psi_1'$ is sign-definite in $[x_n, x_{n+1}]$,
so using \eqref{id1} and $\Delta x_n \ge \Delta x_*$ we get a contradiction
\begin{equation*}
\begin{split}
&2p_*C_* >
|\psi_1(x_{n+1})| + |\psi_1(x_n)| \ge
|\psi_1(x_{n+1}) - \psi_1(x_n)| = \\[0.5ex]
&\qquad\qquad\qquad =\int_{x_n}^{x_{n+1}}|\psi_1'(t)|dt > \Delta x_n (1 - p_*)C_* \ge \Delta x_* (1 - p_*)C_*=2p_*C_*.
\end{split}
\end{equation*}
Thus the point $x_n^*$ with required properties exists.

Since $|\psi_1'(x)|\le C$, for every $x\in[x_n,x_{n+1}]$ such that
$| x - x_n^*| \le \frac{p_*C_*}{2C}$ one has $|\psi_1(x)| \ge \frac{p_*C_*}{2}$.
We have shown that every interval $[x_n,x_{n+1}]$ contains a subinterval of length
$l:=\min\big(\Delta x_*,\tfrac{p_*C_*}{2C}\big)$,
where $|\psi_1(x)| \ge \frac{p_*C_*}{2}$. Therefore
\[
\int_0^{x_n} |\psi_1(t)|^2dt \ge \frac{lp_*^2C_*^2}{4}n.
\]
On the other hand,
\[
\int_0^{x_n} |\psi_1(t)|^2dt \le \Delta x^*C^2n.
\]
Summing up, for every solution $\psi$ the integral $\int_0^{x_n}|\psi(t)|^2dt$ has two-sided linear estimate.
Thus no subordinate solution exists.
\end{proof}

%----------------------------------------------------------------
\subsection{Benzaid-Lutz theorems for discrete linear systems in $\dC^2$}
%----------------------------------------------------------------
\label{sec:BL}
The results of \cite{BL87} translate classical theorems due to
N.~Levinson \cite{L48} and W.~Harris and D. Lutz  \cite{HL75} on the asymptotic
integration of ordinary differential linear systems to the case of discrete linear systems.
The major advantage of these methods is that they allow to reduce under certain assumptions
the asymptotic integration of some general discrete linear systems to the asymptotic integration
of diagonal discrete linear systems. For our applications it is sufficient to formulate Benzaid-Lutz theorems only
for discrete linear systems in $\dC^2$.
The first lemma of this subsection is a direct consequence of ~\cite[Theorem 3.3]{BL87}.

%----------------------------------------------------------------
 \begin{lem}\label{lemma Benzaid-Lutz different modules}
Let $\mu_{\pm}\in\dC\setminus\{0\}$ be such that $|\mu_+|\neq|\mu_-|$
and let $\{V_n\}_{n\in\dN}\in \ell^2(\dN,\dC^{2\times 2})$.
If the coefficient matrix of the discrete linear system
    \begin{equation*}
        \overrightarrow u_{n+1}=
        \left[
        \left(
          \begin{array}{cc}
            \mu_+ & 0 \\
            0 & \mu_- \\
          \end{array}
        \right)
        +V_n
        \right]
        \overrightarrow u_n
    \end{equation*}
is non-degenerate for every $n\in\dN$,
then this system has a basis of solutions
$\overrightarrow u_n^{\pm}$ with the following asymptotics:
        \begin{equation*}
          \begin{split}
            &\overrightarrow u_n^+=
            \left[
            \begin{pmatrix}    1 \\
                0 \\
              \end{pmatrix}
                      +o(1)
            \right]
            \prod\limits_{k=1}^n(\mu_++(V_k)_{11})\quad \text{as}~n\rightarrow\infty,
            \\[0.5ex]
            &\overrightarrow u_n^-=
            \left[
\begin{pmatrix}
               0 \\
                1 \\
              \end{pmatrix}
            +o(1)
            \right]
            \prod\limits_{k=1}^n(\mu_-+(V_k)_{22})\quad\text{as}~n\rightarrow \infty,
\end{split}
        \end{equation*}
 where by $(V_k)_{11}$ and $(V_k)_{22}$ we denote the diagonal entries of the matrices $V_k$, and the factors $(\mu_+ + (V_k)_{11})$ and $(\mu_-+(V_k)_{22})$ should be replaced by 1 for those values of the index $k$ for which they vanish (only a finite number).
   \end{lem}
%----------------------------------------------------------------

The following lemma is a simplification of \cite[Theorem 3.2]{BL87}.

%----------------------------------------------------------------
  \begin{lem}\label{lemma Benzaid-Lutz same modules}
Let
$\{t_n\}_{n\in\mathbb N}\in  \ell^2(\dN; \dR)$ and $\{V_n\}_{n\in\mathbb N}\in\ell^2(\dN;\dC^{2\times 2})$
 be such that $t_n\ge 0$ and that the sum
   $\sum\limits_{n=1}^{\infty}V_n$ is (conditionally) convergent
with
\[
\bigg\{\sum_{k=n}^{\infty}V_k\bigg\}_{n\in\dN}\in
\ell^2(\dN;\dC^{2\times 2}).
\]
If for every $n\in\dN$
\[
 \det
        \left[
        \begin{pmatrix}
            1+t_n & 0 \\
            0 & 1-t_n \\
    \end{pmatrix}
         +V_n
        \right]
        \neq0,
\]
    then the discrete linear system
    \begin{equation*}
        \overrightarrow u_{n+1}=
         \left[
        \begin{pmatrix}
            1+t_n & 0 \\
            0 & 1-t_n \\
    \end{pmatrix} +V_n
        \right]
        \overrightarrow u_n
    \end{equation*}
    has a basis of solutions $\overrightarrow u_n^{\pm}$ with the following asymptotics:
    \begin{equation*}
\begin{split}
      &\overrightarrow u_n^+=
        \left[
        \begin{pmatrix}
            1 \\
            0 \\
          \end{pmatrix}
        +o(1)
        \right]
        \prod\limits_{k=1}^n(1+t_k)\quad \text{as}~n\rightarrow \infty,
        \\[0.5ex]
        &\overrightarrow u_n^-=
        \left[
        \begin{pmatrix}
              0 \\
            1 \\
       \end{pmatrix}
        +o(1)
        \right]
        \prod\limits_{k=1}^n(1-t_k) \quad\text{as}~n\rightarrow\infty,
  \end{split}
  \end{equation*}
 where the factor $(1-t_k)$ should be replaced by 1 for those values of the index $k$ for which it vanishes (only a finite number).
    \end{lem}
%----------------------------------------------------------------

%*************************************************************
\section{Spectral and asymptotic analysis}
%*************************************************************

\subsection{Asymptotic analysis of a special class of discrete linear systems}
\label{sec:class}
In this section we study a special class of discrete linear systems
that encapsulates system \eqref{system for delta} corresponding to $X = \{x_n\colon n\in\dN\}$ and $\alpha = \{\alpha_n\}_{n\in\dN}$ as in Models I and II
described in the introduction.

Let the parameters $l$, $a$, $b$, $\gamma$ and $\omega$ satisfy
the conditions
\begin{equation}
\label{parameters}
l \in \dR,\quad a,b\in\dC, \quad\gamma\in\left(\frac12,1\right],\quad
\text{and}\quad\omega\in(0,\pi)\setminus\left\{{ \pi/2}\right\}.
\end{equation}
For further purposes we define
\begin{equation}
\label{mu}
\mu_\pm(l) :=l\pm\sqrt{l^2-1}
\end{equation}
with any choice of { the} branch of the square root (although we specify it explicitly below in { the} { subcase} $|l|=|\cos\omega|$). Define further
\begin{equation}
\label{betaphi}
z_{\pm}:=\frac{a e^{-i\omega}\pm be^{-2i\omega}}{2i\sin\omega},\qquad
\beta_{\pm} :=|z_{\pm} |,\qquad\varphi_{\pm}  :=\arg (z_{\pm} ),
\end{equation}
and
\begin{equation}
\label{fn}
f^{\pm}_n(\beta) :=
        \left\{
        \begin{array}{ll}
        \exp\left(\pm\frac{\beta n^{1-\gamma}}{1-\gamma}\right),&\text{ if }\gamma<1,
        \\
        n^{\pm\beta},&\text{ if }\gamma=1.
        \end{array}
        \right.
\end{equation}
The following lemma has technical nature and helps to simplify the analysis of cases (Model I and Model II).
%----------Main Lemma------------------------------------------------------------------------
\begin{lem}\label{lemma general}
Let the parameters $l$, $a,b$, $\gamma$ and $\omega$ be as in \eqref{parameters}.
Let $\mu_\pm$, $\beta, \varphi$ and $f_n^\pm(\cdot)$ be as in \eqref{mu}, \eqref{betaphi}
and \eqref{fn}, respectively. Let the sequence of matrices $\{R_n\}_{n\in\dN}\in\ell^1(\dN,\dC^{2\times 2})$ be arbitrary.
If the coefficient matrix of the discrete linear system
 \begin{equation}\label{system on u}
        \overrightarrow u_{n+1}=
        \left[
        \begin{pmatrix}
             0 & 1 \\
            -1 & 2l \\
          \end{pmatrix}
        +
       \begin{pmatrix}
           0 & 0 \\
            a & b \\
        \end{pmatrix}
        \frac{e^{2i\omega n}}{n^{\gamma}}
        +
\begin{pmatrix}
           0 & 0 \\
            \ov a & \ov b \\
\end{pmatrix}
        \frac{e^{-2i\omega n}}{n^{\gamma}}
        +R_n
        \right]
        \overrightarrow u_n
    \end{equation}
is non-degenerate for every $n\in\dN$,
then this system has a basis of solutions $\overrightarrow u_n^{\pm}$ with the following asymptotics.
\begin{itemize}
 \item [\rm (i)] If $|l|\in { [0,+\infty)}\setminus \{|\cos\omega|, 1\}$, then
    \begin{equation*}
    \overrightarrow u_n^{\pm}=
        \left[
        \left(
          \begin{array}{c}
            1 \\
            \mu_{\pm} \\
          \end{array}
        \right)
        +o(1)
        \right]
        \mu_{\pm}^n\quad\text{as}~ n\rightarrow\infty.
    \end{equation*}
  \item[\rm (ii)]  If $l=\cos\omega$, then
    \begin{equation*}
        \overrightarrow u_n^+=
        \left[
        \left(
          \begin{array}{c}
            \cos(\omega n+\varphi_+ /2) \\
            \cos(\omega(n+1)+\varphi_+/2) \\
          \end{array}
        \right)
        +o(1)
        \right]
        f^+_n(\beta_+) \quad \text {as}~n\rightarrow \infty,
    \end{equation*}
    and
    \begin{equation*}
        \overrightarrow u_n^-=
        \left[
        \left(
          \begin{array}{c}
            \sin(\omega n+\varphi_+/2) \\
            \sin(\omega(n+1)+\varphi_+/2) \\
          \end{array}
        \right)
        +o(1)
        \right]
        f^-_n(\beta_+)\quad\text{as}~n\rightarrow \infty.
    \end{equation*}
  \item[\rm (iii)]  If $l=-\cos\omega$, then
    \begin{equation*}
        \overrightarrow u_n^+=(-1)^n
        \left[
        \left(
          \begin{array}{c}
            \cos(\omega n+\varphi_- /2) \\
            -\cos(\omega(n+1)+\varphi_-/2) \\
          \end{array}
        \right)
        +o(1)
        \right]
        f^+_n(\beta_-) \quad \text {as}~n\rightarrow \infty,
    \end{equation*}
    and
    \begin{equation*}
        \overrightarrow u_n^-=(-1)^n
        \left[
        \left(
          \begin{array}{c}
            \sin(\omega n+\varphi_-/2) \\
            -\sin(\omega(n+1)+\varphi_-/2) \\
          \end{array}
        \right)
        +o(1)
        \right]
        f^-_n(\beta_-)\quad\text{as}~n\rightarrow \infty.
    \end{equation*}
\end{itemize}
 \end{lem}

\begin{rem}
\label{rem1}
We do not consider the (double-root) case $|l|  = 1$ .
The analysis in this special case is technically involved
and we refer the reader to \cite{J06, JNSh07,NS10}.\
\end{rem}

\begin{proof}[Proof of Lemma~\ref{lemma general}]
Since $|l| \in { [0,+\infty)}\setminus  \{1\}$,  the constant term in the coefficient matrix in~\eqref{system on u}
can be diagonalized as follows
\begin{equation*}
\begin{pmatrix}
   1 & 1 \\
  \mu_+ & \mu_- \\
    \end{pmatrix}^{-1}
\begin{pmatrix}
    0 & 1 \\
      -1 & 2l \\
     \end{pmatrix}
     \begin{pmatrix}
    1 & 1 \\
        \mu_+ & \mu_- \\
      \end{pmatrix}
    =
   \begin{pmatrix}
        \mu_+ & 0 \\
        0 & \mu_- \\
\end{pmatrix}.
\end{equation*}
In view of the identity
\[
\begin{pmatrix}
    1 & 1 \\
        \mu_+ & \mu_- \\
   \end{pmatrix}^{-1}
=
\frac1{\mu_--\mu_+}
\begin{pmatrix}
         \mu_- & -1 \\
        -\mu_+ & 1 \\
      \end{pmatrix},
\]
the substitution
\begin{equation}
\label{substitution1}
    \overrightarrow u_n=
    \left(
      \begin{array}{cc}
        1 & 1 \\
        \mu_+ & \mu_- \\
      \end{array}
    \right)
    \overrightarrow{v}_n
\end{equation}
transforms the system \eqref{system on u} on $\overrightarrow u_n$
into the system on $\overrightarrow v_n$ given below
\begin{equation}
\label{system on v}
\begin{split}
 \overrightarrow{v}_{n+1}=&
 \Bigg[
 \begin{pmatrix}
 \mu_+ & 0 \\
   0        & \mu_- \\
  \end{pmatrix}
 +
 \frac{1}{\mu_- -\mu_+}
\begin{pmatrix}
-(a+\mu_+b) & -(a+\mu_-b) \\
(a+\mu_+b) & (a+\mu_-b) \\
\end{pmatrix}
\frac{e^{2i\omega n}}{n^{\gamma}} \\[0.5ex]
&\qquad +
\frac{1}{\mu_- -\mu_+}
\begin{pmatrix}
-(\ov a+\mu_+\ov b) & -( \ov a+\mu_-  \ov b) \\
(\ov a+\mu_+ \ov b) & (\ov a+\mu_- \ov b) \\
\end{pmatrix}
\frac{e^{-2i\omega n}}{n^{\gamma}}
    +R_n^{(2)} \Bigg]   \overrightarrow{ v}_n
\end{split}
\end{equation}
with some $\{R_n^{(2)}\}_{n\in\mathbb N}\in \ell^1(\mathbb N, \mathbb C^{2\times2})$.

\noindent (i) The case $|l| \in { [0,+\infty)}\setminus\{|\cos\omega|, 1\}$ splits into two subcases:
 $| l| > 1$ and $|l|\in [0,1)\setminus \{|\cos\omega|\}$.
The condition $|l|>1$ implies 
{ $\mu_+,\mu_- \in\dR\setminus\{0\}$ and
$|\mu_-| \ne |\mu_+|$}. Thus Lemma~\ref{lemma Benzaid-Lutz different modules}
is applicable to the system \eqref{system on v}  and it  gives us a basis. Reverting the substitution \eqref{substitution1} we get the statement.
The condition $|l|\in[0,1)\setminus\{|\cos\omega|\}$ implies  $|\mu_+| = |\mu_-|=1$.
We reduce this situation to Lemma \ref{lemma Benzaid-Lutz same modules}
by the following substitution:
\begin{equation}
\label{substitution2}
\overrightarrow{v}_n=
    \begin{pmatrix}
        \mu_+^n & 0 \\
        0 & \mu_-^n \\
      \end{pmatrix}
    \overrightarrow w_n.
\end{equation}
The system on $\overrightarrow w_n$ has the following form
\begin{equation}
\label{system on w}
\begin{split}
\overrightarrow w_{n+1}&\!=\!
    \Bigg[
    I\!+\!\frac{1}{\mu_- -\mu_+}
      \begin{pmatrix}
     -\mu_-( a\!+\!\mu_+ b) & -\mu_-( a\!+\!\mu_- b)\mu_-^{2n} \\
        \mu_+(a\!+\!\mu_+b)\mu_+^{2n} & \mu_+(a\!+\!\mu_- b) \\
      \end{pmatrix}
    \frac{e^{2i\omega n}}{n^{\gamma}}\\
    &\qquad +\frac{1}{\mu_--\mu_+}
    \begin{pmatrix}
        - \mu_-(\ov a+\mu_+\ov b) & -\mu_-(\ov a+\mu_-\ov b)\mu_-^{2n} \\
        \mu_+(\ov  a+\mu_+\ov b)\mu_+^{2n} & \mu_+(\ov a+\mu_- \ov b) \\
      \end{pmatrix}
    \frac{e^{-2i\omega n}}{n^{\gamma}}
    +R_n^{(3)}
    \Bigg]
    \overrightarrow w_n
\end{split}
\end{equation}
with some $\{R_n^{(3)}\}_{n\in\mathbb N}\in \ell^1(\mathbb N, \mathbb C^{2\times2})$.
We have used that $\mu_+\mu_- = 1$.
Now Lemma~\ref{lemma Benzaid-Lutz same modules} is applicable
 with $t_n = 0$ and
reverting substitutions \eqref{substitution1} and \eqref{substitution2} we get the statement.

\noindent (ii)
We start with the system \eqref{system on w}. In the case
that  $l = \cos\omega$,
Lemma~\ref{lemma Benzaid-Lutz same modules} can not be applied immediately.
In order to bring the system into the form in which Lemma~\ref{lemma Benzaid-Lutz same modules}
is applicable we make further transformations. Since $l = \cos\omega$, we clearly have  $\mu_{\pm}=e^{\pm i\omega}$,
and thus grouping summands in the right way we can rewrite our system in the form
\begin{equation*}
    \overrightarrow w_{n+1}=
    \left[
    I+\begin{pmatrix}
        0 & z_+ \\
       \ov{z_+}& 0 \\
      \end{pmatrix}
    \frac1{n^{\gamma}}+V_n
    \right]
    \overrightarrow w_n,
\end{equation*}
with $z_+=\beta_+ e^{i\varphi_+}$ defined in \eqref{betaphi}, and with
\begin{equation*}
\begin{split}
V_n :=& \frac{1}{(\mu_- -\mu_+)n^\gamma}\Bigg[
      \begin{pmatrix}
     -\mu_-( a+\mu_+ b) & 0\\
        \mu_+(a+\mu_+b)\mu_+^{2n} & \mu_+(a+\mu_- b) \\
      \end{pmatrix}
    e^{2i\omega n}\\[0.6ex]
     &\qquad\qquad\qquad +
    \begin{pmatrix}
        - \mu_-(\ov a+\mu_+\ov b) & -\mu_-(\ov a+\mu_-\ov b)\mu_-^{2n} \\
        0 & \mu_+(\ov a+\mu_- \ov b) \\
      \end{pmatrix}
    e^{-2i\omega n}\Bigg]
    +R_n^{(3)}.
\end{split}
\end{equation*}
The sequence $\{V_n\}_{n\in\dN}$ satisfies  the conditions of Lemma~\ref{lemma Benzaid-Lutz same modules}. Since
\begin{equation*}
\begin{split}
      \begin{pmatrix}
        e^{i\varphi_+/2} & \frac{e^{i\varphi_+/2}}{2i} \\
        e^{-i\varphi_+/2} & -\frac{e^{-i\varphi_+/2}}{2i}\\
        \end{pmatrix}^{-1}
        \begin{pmatrix}
          0 & \beta_+ e^{i\varphi_+} \\
          \beta_+ e^{-i\varphi_+} & 0 \\
        \end{pmatrix}
      \begin{pmatrix}
        e^{i\varphi_+/2} & \frac{e^{i\varphi_+/2}}{2i} \\
        e^{-i\varphi_+/2} & -\frac{e^{-i\varphi_+/2}}{2i}\\
      \end{pmatrix}
      \\
  =\begin{pmatrix}
     \beta_+ & 0 \\
       0 & -\beta_+ \\
    \end{pmatrix},
\end{split}
\end{equation*}
the substitution
\begin{equation}
\label{substitution3}
    \overrightarrow w_n=
      \begin{pmatrix}
        e^{i\varphi_+/2} & \frac{e^{i\varphi_+/2}}{2i} \\
        e^{-i\varphi_+/2} & -\frac{e^{-i\varphi_+/2}}{2i}\\
      \end{pmatrix}
    \overrightarrow{x}_n
\end{equation}
leads to the system on $\overrightarrow{x}_n$ of the form
\begin{equation}
\label{system on x}
    \overrightarrow{x}_{n+1}
    =
    \left[
    I+
    \left(
      \begin{array}{cc}
        \beta_+ & 0 \\
        0 & -\beta_+ \\
      \end{array}
    \right)\frac{1}{n^\gamma}
    +\widetilde V_n
    \right]
    \overrightarrow{x}_n,
\end{equation}
for which Lemma \ref{lemma Benzaid-Lutz same modules}
is applicable with $t_n = \frac{\beta_+}{n^\gamma}$.

For sufficiently  large $n_0$
\begin{equation}
\label{product}
    \prod\limits_{k=n_0}^n\left(1\pm\frac{\beta_+}{k^{\gamma}}\right)
    \sim \text{const}\cdot f_n^\pm(\beta_+),\qquad n\rightarrow\infty,
\end{equation}
with $f_n^\pm(\cdot)$ defined in \eqref{fn}.
Now we apply Lemma \ref{lemma Benzaid-Lutz same modules}
to the system \eqref{system on x} and
using \eqref{product} we get a basis of solutions of that system
of the form
\[
\overrightarrow x_n^+ = \Bigg[\begin{pmatrix} 1\\ 0\end{pmatrix} + o(1)\Bigg] f_n^+(\beta_+),\quad\overrightarrow x_n^- = \Bigg[\begin{pmatrix} 0\\1\end{pmatrix} + o(1)\Bigg] f_n^-(\beta_+),\qquad n\rightarrow\infty.
\]
%The above asymptotics imply the statement immediately
%after reverting substitutions
%\eqref{substitution1}, \eqref{substitution2},  \eqref{substitution3}
%and coming back to the sequence $\overrightarrow u_n$.
After substituting this into \eqref{substitution3} one gets
a basis of solutions for the system \eqref{system on w}
of the form
\[
\overrightarrow w_n^+ =
\Bigg[
\begin{pmatrix}
e^{i\varphi_+/2}\\
 e^{-i\varphi_+/2} \end{pmatrix} + o(1)
\Bigg]
f_n^+(\beta_+),\quad
\overrightarrow w_n^- =
\Bigg[
\begin{pmatrix}
\frac{e^{i\varphi_+/2}}{2i} \\
-\frac{e^{-i\varphi_+/2}}{2i}\\
\end{pmatrix} +o(1)
\Bigg]
f_n^-(\beta_+), \quad n\rightarrow\infty.
\]
Substituting this into \eqref{substitution2},
one obtains a basis of solutions for the system \eqref{system on v}
of the form
\[
 \begin{split}
\overrightarrow v_n^+ =&
\Bigg[
\begin{pmatrix}
e^{i(\varphi_+/2 + \omega n)}\\
e^{-i(\varphi_+/2 + \omega n)}
\end{pmatrix}
+ o(1)\Bigg]
f_n^+(\beta_+),\qquad n\rightarrow\infty,\\[0.3ex]
\overrightarrow v_n^- =&
\Bigg[
\begin{pmatrix}
\frac{e^{i(\varphi_+/2+ \omega n)}}{2i} \\
-\frac{e^{-i(\varphi_+/2+ \omega n)}}{2i}\\
\end{pmatrix}
+ o(1)\Bigg]
f_n^-(\beta_+), \qquad n\rightarrow\infty.
 \end{split}
\]
Finally, substituting into \eqref{substitution1} we
get the claim.

\noindent (iii)
One can literally repeat the proof of item (ii) taking $\mu_{\pm}=-e^{\pm i\omega}$  and { replacing} $z_+,\beta_+$ and $\varphi_+$ by $z_-,\beta_-$ and $\varphi_-$, respectively.

\end{proof}

\subsection{Decomposition of the transfer matrices}
\label{sec:decomp}

Recall that $T_n(k)$ denotes the transfer matrix \eqref{T-n-k}.
In this section we decompose this transfer matrix  subject
to Model I and Model II in the form
\begin{equation}
\label{Tnk}
     T_n(k) =   \begin{pmatrix}
             0 & 1 \\
            -1 & 2l \\
          \end{pmatrix}
        +
       \begin{pmatrix}
           0 & 0 \\
            a & b \\
        \end{pmatrix}
        \frac{e^{2i\omega n}}{n^{\gamma}}
        +
\begin{pmatrix}
           0 & 0 \\
            \ov a & \ov b \\
\end{pmatrix}
        \frac{e^{-2i\omega n}}{n^{\gamma}}
        +R_n,
\end{equation}
where $R_n\in\ell^1(\dN;\dC^{2\times 2})$. This decomposition
allows to apply later Lemma~\ref{lemma general}. In order
{ not to} bore the reader with these long, but straightforward calculations we put them into the appendix.

\noindent {\em Model I: Amplitude perturbation.}
The sequence of strengths $\wt\alpha = \{\wt \alpha_n\}_{n\in\dN}$ and the discrete set
$\wt X = \{\wt x_n\colon n\in\dN\}$ are defined as follows:
\begin{equation*}
\label{amplitude}
    \wt\alpha_n:=\alpha_0+\frac{c\sin(2\omega n+\phi)}{n^{\gamma}}+q_n,\quad    \wt x_n:=nd \quad\text{and}\quad     \wt x_{n+1}- \wt  x_n\equiv d,
\end{equation*}
with $d, \alpha_0, c,\omega,{ \phi},\gamma, \{q_n\}_{n\in\dN}$ as in \eqref{conditions on constants}. Substitution of these expressions into \eqref{T-n-k} and calculation show that $T_n(k)$ for $k\notin\pi{\dN_0}/d$
has the representation of the type \eqref{Tnk} with the following values of the parameters:
\begin{equation}
\label{ab1}
\begin{split}
& l = \cos(kd) + \frac{\alpha_0 \sin(kd)}{2k} = { L}(k),\\[0.5ex]
& a =0 ,\qquad  b =  \frac{c\sin(kd)}{2ik}e^{i\phi}.
\end{split}
\end{equation}
{In the case $\lambda=0$ these expressions should be understood as their limits as $k\rightarrow0$, i.e.,
\begin{equation}
\label{ab10}
\begin{split}
l=&1+\frac{\alpha_0d}2= { L}(0),\\[0.5ex]
a=&0 ,\qquad b=\frac{cd}{2i}e^{i\phi}.
\end{split}
\end{equation}

\noindent {\em Model II: Positional perturbation.}
The sequence of strengths $\wh\alpha = \{\wh \alpha_n\}_{n\in\dN}$
and the discrete set $\wh X = \{\wh x_n\colon n\in\dN\}$ are defined as follows:
\begin{equation*}
\label{position}
\wh \alpha_n\equiv\alpha_0,\qquad \wh x_n:= nd+\frac{c\sin(2\omega n+\phi)}{n^{\gamma}}+q_n,
\end{equation*}
 with $d, \alpha_0, c,\omega,{ \phi},\gamma, \{q_n\}_{n\in\dN}$ as in \eqref{conditions on constants}.
In this case substitution into \eqref{T-n-k} and a tedious calculation give for $k\notin\pi{ \dN_0}/d$
the representation of the type \eqref{Tnk} with the following values of the parameters:
\begin{equation}
\label{ab2}
\begin{split}
 l  = & \cos(kd) + \frac{\alpha_0 \sin(kd)}{2k} = { L}(k),\\[0.5ex]
 a = & -2ikc\cot(kd) \sin^2(\omega)e^{i\phi},\\[0.5ex]
 b = & c\cos(kd)\sin(\omega)\Big[4k\cos(\omega)\cot(2kd) - 2k\cot(kd)e^{-i\omega} +\alpha_0e^{i\omega}\Big]e^{i\phi}.
\end{split}
\end{equation}
{Again, in the case $\lambda=0$ these expressions should be understood as their limits as $k\rightarrow0$, i.e.,
\begin{equation}
\label{ab20}
\begin{split}
l=&1+\frac{\alpha_0d}2= { L}(0),\\[0.5ex]
a=&-2ic\sin^2(\omega)e^{i\phi}/d,\\[0.5ex]
b=&c\sin(\omega)[2\cos(\omega)/d-2e^{-i\omega}/d+\alpha_0e^{i\omega}]e^{i\phi}.
\end{split}
\end{equation}}

%
%\rm III. {\em $\delta'$-Kronig-Penney model with a Wigner-von Neumann amplitude disorder.}
%It follows from \eqref{spectral equation for delta'} that $M_n'(k)$
%has the representation of the type \eqref{system on u} with
%\begin{equation}
%\label{ab3}
%\begin{split}
%&l  =  \cos(kd) - \alpha_0 \frac{k \sin(kd)}{2} = L_{\delta'}(k),\\[0.5ex]
%&a = 0,\qquad b  =  -\frac{ck\cos(kd)}{2i}.
%\end{split}
%\end{equation}
%
%\rm IV. {\em $\delta'$-Kronig-Penney model with a Wigner-von Neumann position disorder.}
%Combining \eqref{spectral equation for delta'} and \eqref{position} we get after tedious calculations
%that $M_n'(k)$ has the representation of the type \eqref{system on u} with
%\begin{equation}
%\label{ab4}
%\begin{split}
%l  = & \cos(kd) - \alpha_0 \frac{k \sin(kd)}{2} = L_{\delta'}(k),\\[0.5ex]
%a = & -2ikc\cot(kd) \sin^2(\omega),\\[0.5ex]
%b = & c\cos(kd)\sin(\omega)\Big[4k\cos(\omega)\cot(2kd) - 2k\cot(kd)e^{-i\omega} - \alpha_0k^2e^{i\omega}\Big].
%\end{split}
%\end{equation}
%

\subsection{Asymptotics of generalized eigenvectors for the spectral bands}
\label{sec:main}
In this section we obtain
using Lemma~\ref{lemma general} asymptotics of  solutions $\xi$
of the difference equation \eqref{equation for xi}
in the cases of Model I and Model II.
Using the statements from Section~\ref{sec:subord}
we come from the asymptotics of these solutions to the conclusions
about the spectral bands and the critical points inside them.

In the table below we list certain functions playing a role in the main theorem of this section.
\begin{center}
\begin{table}[h]
\begin{tabular}{|c| c | c |}
\hline
%\backslashbox{Funct.:}{Perturb.:}
&Model I& Model II\\
\hline
&&\\
$\beta(k) := $& $\Big|\frac{c\sin(kd)}{4k\sin\omega}\Big|$ & $\big|\frac{c\alpha_0}{2}\big|$\\
&&\\
\hline
&&\\
$\vartheta_\pm(k) := $ & $\frac{1}{2}\big(\arg\big(\mp \frac{c\sin(kd)}{k}\big)+\phi\big)$  & $\frac12\big(\arg(ic\alpha_0)+\phi\big)$\\
&&\\
\hline
\end{tabular}
\vspace{3mm}
\caption{Notation for Theorem~\ref{thm:asymp1}.}
\label{table}
\end{table}
\vspace{-3mm}
\end{center}
As before, in the case $\lambda=0$ the expressions are understood as their limits as $k\rightarrow0$, i.e., for the Model I { we get} the following:
\begin{equation*}
\beta(0)=\left|\frac{cd}{4\sin\omega}\right|,\qquad \vartheta_{\pm}(0)=\frac12(\arg(\mp cd)+\phi),
\end{equation*}
and for the Model II the same as in the table.

%\begin{center}
%\begin{table}[H]
%\label{table}
%\caption{Notation for Theorems~\ref{thm:asymp1} and~\ref{thm:asymp2}}
%\begin{tabular}{| c | c | c | c | c |}
%    \hline &&&&\\
% Function   & $\delta$-amp. & $\delta$-pos. & $\delta'$-amp. & $\delta'$-pos. \\
%&&&&\\
%    \hline
%&&&&\\
%    $\beta(k) :=$ & $\frac{c|\sin(kd)|}{4k\sin\omega}$ & $\frac{c\alpha_0}2$
%    &
%    $\frac{ck|\sin(kd)|}{4\sin\omega}$ & $\frac{c\alpha_0k^2}2$\\
%&&&&\\
%    \hline
%&&&&\\
%    $\theta(k) :=$ & $-\frac{\arg(i\sin(kd))}2$  & $\frac{\pi}4$ & $\frac{\arg(i\sin(kd))}{2}$ & $-\frac{\pi}4$ \\
%&&&&\\    \hline
%\end{tabular}
%
%\end{table}
%\end{center}

%---------Asymptotics in the delta-case-------------------------------------------------------
\begin{thm}
\label{thm:asymp1}
Let the  parameters $d$, $\alpha_0$,  $c$, $\omega$,  $\gamma$, { $\phi$}, and $\{q_n\}_{n\in\dN}$
be as in \eqref{conditions on constants}.
Let the function ${ L(\cdot)}$ be as in \eqref{Lyapunov function for delta}.
Assume either that $X =\wt X$,  $\alpha = \wt \alpha$ as in Model I or that $X = \wh X$, $\alpha = \wh \alpha$  as
in Model II. Then for
$\lambda \in \dR$ with $k := \sqrt{\lambda}$ satisfying
$|{ L}(k)|\ne 1$ the finite difference equation \eqref{equation for xi}  has a basis of  solutions
$\xi^{\pm}(k) = \{\xi_n^\pm(k)\}_{n \ge N(k)}$ with the following asymptotics:
\begin{itemize}\setlength{\itemsep}{1.2ex}

\item[\rm (i)] If $\big|{ L}(k)\big|\in { [0,+\infty)}\setminus\{|\cos\omega|,1\}$, then
\begin{equation*}
\xi_n^{\pm}(k)=\bigg({ L}(k)\pm\sqrt{{ (L(k))^2}-1}\bigg)^n\big(1+o(1)\big)
\quad\text{as}~n\rightarrow\infty.
\end{equation*}

\item[\rm (ii)]
If ${ L}(k)=\cos\omega$, then
\begin{equation*}
\begin{split}
\xi^+_n(k) =\big(\cos(\omega n+\vartheta_+(k))+o(1)\big)f^+_n(\beta(k)) \quad\text{as}~n\rightarrow\infty,\\[0.5ex]
\xi^-_n(k) = \big(\sin(\omega n+\vartheta_+(k))+o(1)\big)f^-_n(\beta(k)) \quad \text{as}~n\rightarrow\infty.
\end{split}
\end{equation*}
If ${ L}(k)= - \cos\omega$, then
\begin{equation*}
\begin{split}
\xi^+_n(k) =(-1)^{n}\big(\cos(\omega n+\vartheta_-(k))+o(1)\big)f^+_n(\beta(k))\quad\text{as}~n\rightarrow\infty,\\[0.5ex]
\xi^-_n(k) = (-1)^{n}\big(\sin(\omega n+ \vartheta_-(k))+o(1)\big)f^-_n(\beta(k)) \quad \text{as}~n\rightarrow\infty.
\end{split}
\end{equation*}
Here $f_n^\pm(\cdot)$ are defined in \eqref{fn} and the functions  $\beta(\cdot)$ and $\vartheta_{\pm}(\cdot)$ are
given in the 1st column of the Table \ref{table} in the case of Model I and
in the 2nd column in the case of Model II, respectively.
\end{itemize}
\end{thm}
\begin{proof}
The transfer matrix $T_n(k)$ of the discrete linear system \eqref{system for delta}
decomposes into the form in \eqref{system on u}
with the parameters $a$, $b$ and $l$ given in~\eqref{ab1}  and \eqref{ab10} for the Model I
and given in~\eqref{ab2}  and \eqref{ab20} for the Model II, respectively. Note also
that $T_n(k)$ is non-degenerate for all  $n$ large enough.

In the case $\big|{ L}(k)\big| \in { [0,+\infty)}\setminus \{|\cos\omega|,1\}$
Lemma~\ref{lemma general}\,(i) could be applied and it automatically implies the item (i).
Assume now ${ L}(k) = \pm \cos\omega$. In the case of Model I we plug into the formula
\eqref{betaphi} the values $a$ and $b$ from \eqref{ab1} or \eqref{ab10} and we get the following
\[
 z_+(k)=-z_-(k)=-\frac{c\sin(kd)e^{i(\phi-2\omega)}}{4k\sin\omega},
  \quad{
 \Bigg(
z_+(0) = -z_-(0) =
 -\frac{cde^{i(\phi-2\omega)}}{4\sin\omega}\Bigg),}
\]
see  the { appendix}. In the case of Model II we plug into the formula \eqref{betaphi}
the values  $a$ and $b$  from \eqref{ab2} or \eqref{ab20} and using that
$L(k) = \pm\cos\omega$
after a long and tedious calculation (provided in the appendix) we get
\[
 z_+(k) = z_-(k) = \frac{ic\alpha_0e^{i(\phi-2\omega)}}{2}
\]
Now Lemma~\ref{lemma general}\,(ii) could be applied.
It gives us the asymptotics for the sequences of $\dC^2$-vectors
\[
\overrightarrow u_n^\pm(k)=
\begin{pmatrix}
\xi_{n-1}^\pm(k) \\
\xi_n^\pm(k)
\end{pmatrix}.
\]
Extracting the second components from these asymptotics we obtain the statement.
\end{proof}

Given the asymptotics, we can now come to the conclusions about the structure
of the absolutely continuous spectrum using the subordinacy theory,
see Section~\ref{sec:subord}.
Let the parameters $d$, $\alpha_0, c$, $\omega$, $\phi$  and  $\gamma$ be as in
\eqref{conditions on constants}. Recall that the self-adjoint operator
$H_{\varkappa}$ corresponds to the half-line Kronig-Penney  model without perturbation,
where the  distances between interaction   centers are constant and equal to $d$
and the strengths of interactions are constant and equal to $\alpha_0$.
The spectral properties of this operator are discussed in the introduction.
The absolutely continuous spectrum of the operator $H_{\varkappa}$ is the set
\begin{equation*}
\label{set}
\sigma_{\rm ac}(H_{\varkappa}) :=\Big\{\lambda\in { \dR}\colon { L}\big(\sqrt{\lambda}\big)\in[-1,1]\Big\}
\end{equation*}
with  the interior part
\[
{\rm Int}\big(\sigma_{\rm ac}(H_{\varkappa})\big)
=\Big\{\lambda\in { \dR}\colon { L}\big(\sqrt{\lambda}\big)\in(-1,1)\Big\}.
\]
Denote the set of all critical points by
\[
\frs_{\rm cr} :=
\Big\{\lambda\in { \dR}\colon { L}\big(\sqrt{\lambda}\big) =
\pm \cos\omega \Big\} \subset {\rm Int}\big(\sigma_{\rm ac}(H_{\varkappa})\big).
\]
\begin{cor}
\label{cor:main}
Let the parameters $d$, $\alpha_0, c$, $\omega$,  $\gamma$, $\phi$  and $\{q_n\}$
be as in \eqref{conditions on constants}, let $\varkappa\in[0,\pi)$,  and let
the sequences $\widetilde X, \widetilde\alpha$ and $\wh X, \wh\alpha$
be as in \eqref{wt X}, \eqref{wt alpha} and in \eqref{wh X}, \eqref{wh alpha}, respectively. Then
the following statements hold.
\begin{itemize}
\item [\rm (i)]
The spectrum of the operators $H_{\varkappa,\wt X,\wt\alpha}$
and $H_{\varkappa,\wh X,\wh\alpha}$
is purely absolutely continuous on the set
\[
{\rm Int}\big(\sigma_{\rm ac}(H_{\varkappa})\big)\setminus\frs_{\rm cr}.
\]
\item[\rm (ii)]
Let  $\lambda\in\frs_{\rm cr}$ and set $k := \sqrt{\lambda}$.
If either $\gamma < 1$ or
\[
\gamma=1\quad\text{and}\quad\Bigg|\frac{c\sin(kd)}{2k\sin\omega}\Bigg| > 1 \quad\left(\text{or, in the case $\lambda=0$},  \Bigg|\frac{cd}{2\sin\omega}\Bigg| > 1\right),
\]
then there exists a unique value $\varkappa$ such that $\lambda$  is an embedded eigenvalue
of $H_{\varkappa, \wt X, \wt\alpha}$.
\item[\rm (iii)]{
If either $\gamma < 1$ or
\[
\gamma=1\quad\text{and}\quad |c\alpha_0|>1,
\]
then there exists a unique value $\varkappa$ such that  $\lambda\in\frs_{\rm cr}$  is an embedded eigenvalue
of $H_{\varkappa, \wh X, \wh\alpha}$.}
\end{itemize}
\end{cor}
\begin{proof}
First of all note that for $k^2\in{\rm Int}\,\sigma_{\rm ac}(H_\varkappa)$,  $k\ne 0$, holds $\liminf_{n\rightarrow\infty}\big|s_n(k) \big |>0$ (for both models).

In the case that {  $k^2\in{\rm Int}
\big(\sigma_{\rm ac}(H_{\varkappa})\big)\setminus  \frs_{\rm cr}$} from the asymptotics of $\{\xi^\pm_n(k)\}_{n \ge N(k)}$ in Theorem~\ref{thm:asymp1}\,(i)
and the formula \eqref{psi from xi}  it follows that all solutions of the spectral equations for both operators $H_{\varkappa,\wt X,\wt\alpha}$
and $H_{\varkappa,\wh X,\wh\alpha}$ are bounded { sufficiently far away from the origin}.
It is also clear that all these
solutions are everywhere bounded,
since they are continuous.
The assertion of the item (i) follows now from Proposition~\ref{prop:stolz}, Proposition~\ref{prop:Stolz2} and boundedness of all solutions.

Conditions of the item (ii) and Theorem~\ref{thm:asymp1}\,(ii) guarantee that the finite
difference equation~\eqref{equation for xi} (or~\eqref{equation for xi 0}, if $\lambda=0$) with $X = \wt X$ and $\alpha =
\wt\alpha$ has the unique square-summable
solution $\{\xi_n^-(k)\}_{n \ge N(k)}$.
This solution can be lifted via the formula
\eqref{psi from xi} (\eqref{psi from xi 0}, respectively) uniquely up to $\wt\psi\in L^2(\wt x_{N(k)},+\infty)$ satisfying $-\wt\psi''(x)=\lambda\wt\psi(x)$ for all
$x\in (\wt x_{N(k)},+\infty)\setminus \wt X$
and interface $\delta$-boundary conditions at
the points $\{\wt x_n\}_{n > N(k)}$.
Clearly enough, one can continue $\wt\psi$ in the unique way
to the left up to $\psi\in L^2(\dR_+)$,
which satisfies \eqref{spectral equation for delta} with $X = \wt X$
and $\alpha = \wt\alpha$.
This fact means that there is a unique value of the boundary parameter $\varkappa$ such that this  $\psi\in L^2(\dR_+)$ satisfies the boundary condition  at the origin,
and thus in this case $\lambda$ is an eigenvalue { of $H_{\varkappa, \wt X,\wt \alpha}$}.

 The proof of item (iii) goes along the same steps as the proof of (ii). One has only to replace $\wt X$ and $\wt\alpha$ by $\wh X$ and
$\wh\alpha$, respectively.
\end{proof}

\subsection{Spectrum in gaps}
%----------------------------------------------------------------
\label{sec:compact}
In this section we show that the spectrum
of the operators  $H_{\varkappa,\wt X,\wt \alpha}$
and $H_{\varkappa,\wh X,\wh \alpha}$ in $\dR\setminus\sigma_{\rm ac}(H_\varkappa)$  is discrete.
This follows from Proposition \ref{prop:compact} below
which is of certain independent interest
and which generalizes also  \cite[Theorem III.2.6.2]{AGHH05} and \cite[Theorem 1]{M95}.
Before formulating  and proving this proposition we provide
some of the required definitions and notations.
Let the discrete set $X = \{x_n\colon n\in\dN\}$ satisfy
the condition~\eqref{condition on X}.
Let $\alpha\in\ell^\infty(\dN;\dR)$ be fixed.
It is known that the self-adjoint   operator with the Neumann boundary condition at the origin  $H_{X,\alpha,\frac{\pi}{2}}$
defined in  Section~\ref{sec:def}
corresponds to the sesquilinear form
\begin{equation}
\label{fra}
\fra_{X,\alpha}[u,v] := (u',v')_{L^2(\dR_+)} +
\sum_{n=1}^\infty \alpha_n u(x_n)\ov{v(x_n)},\quad
\dom \fra_{X,\alpha} := H^1(\dR_+),
\end{equation}
via the first representation theorem, cf.
\cite[Lemma III.1]{AKM10}.
In what follows we keep shorthand notations $(\cdot,\cdot)_{L^2}$
and $(\cdot,\cdot)_{\ell^2}$ for the scalar products
in $L^2(\dR_+)$ and $\ell^2(\dN)$, respectively.

%--------------------------------------------------------------
\begin{prop}
\label{prop:compact}
Let the discrete sets
$X = \{x_n\}_{n\in\dN}$ and $X' = \{x_n'\}_{n\in\dN}$
both satisfy the condition \eqref{condition on X}
and also $x_n -x_n'\rightarrow 0$.
Assume that the interaction strengths
$\alpha, \alpha' \in \ell^\infty(\dN;\dR)$ satisfy
$\alpha_n - \alpha_n'\rightarrow 0$.
Let the self-adjoint operators
$H :=  H_{X,\alpha,\varkappa}$ and
$H' := H_{X',\alpha',\varkappa}$
be as in  Section~\ref{sec:def}.
Then their resolvent difference
\begin{equation}
\label{resdiff}
(H -\lambda)^{-1} - (H' -\lambda)^{-1}
\end{equation}
is compact for all $\lambda\in \rho(H)\cap\rho(H')$,
and, in particular, $\sess(H) = \sess(H')$ holds.
\end{prop}
%--------------------------------------------------------------

\begin{proof}
Without loss of generality we assume that $\varkappa = \frac{\pi}{2}$
since variation of $\varkappa$ leads to rank-one perturbations
of the underlying operators and does not affect  compactness
of the resolvent differences in \eqref{resdiff}.

%--------------------------------------------------------------
\noindent
{\em Step 1.}
Consider two linear bounded mappings
\begin{equation*}
\label{tautau'}
\begin{split}
\tau\colon H^1(\dR_+)\rightarrow \ell^2(\dN),& \qquad
\tau f := \{f(x_n)\}_{n\in\dN},\\
\tau'\colon H^1(\dR_+)\rightarrow \ell^2(\dN),& \qquad
\tau' f := \{f(x_n')\}_{n\in\dN}.
\end{split}
\end{equation*}
Note that boundedness of the mappings $\tau$ and $\tau'$
is implicitly shown in the proof of \cite[Lemma III.1]{AKM10}.
Denote by $I_n$ the interval between the points $x_n$ and $x_n'$
with  the length $|I_n| = |x_n - x_n'|$. Employing Cauchy-Schwarz inequality we obtain
\begin{equation}
\label{CS}
\big|f(x_n) - f(x_n')\big|^2  =
\Bigg|\int_{I_n} f'(x){\rm d} x\Bigg|^2\le
|I_n|\int_{I_n} |f'(x)|^2.
\end{equation}
Let $P_N$
be the orthogonal projection in $\ell^2(\dN)$ onto first $N$ elements.
Decompose
\[
\tau - \tau' =  \tau_N^{<} + \tau_N^{>},
\]
where $\tau_N^{<} := P_N(\tau - \tau')$ and
$\tau_N^{>} := (I - P_N)(\tau - \tau')$.
The mapping
$\tau_N^{<}$ is compact for any $N\in\dN$ due to compactness
of the operator $P_N$.  For sufficiently
large $N$ the intervals $\{I_n\}_{n> N}$ are mutually disjoint.
Thus, in view of \eqref{CS} we get
for sufficiently large $N$ and
arbitrary $f\in H^1(\dR_+)$
\[
\|\tau_N^{>}f\|^2_{\ell^2} =
\sum_{n=N+1}^\infty|f(x_n) -f(x_n')|^2 \le
\sum_{n=N+1}^\infty|I_n|\int_{I_n}|f'(x)|^2.
\]
Hence, we arrive at
\[
\|\tau_N^{>}\|\le\sup_{n >N} \sqrt{|I_n|}\rightarrow 0,\qquad N\rightarrow \infty.
\]
Thus, by \cite[Chapter III, Theorem 4.7]{Kato}
the mapping $\tau - \tau'$ is compact.

%--------------------------------------------------------------
\noindent
 {\em Step 2.}
The operators
$H$ and $H'$ are semibounded from below since
they represent semibounded sequilinear forms.
Hence we can fix a constant $a > 0$
such that $H + a > 0$ and $H' + a > 0$.
We denote $W:=(H+a)^{-1}-(H'+a)^{-1}$. Let $f,g\in L^2(\dR_+)$ and set
\begin{equation}
\label{uv}
u := (H +a)^{-1}f,\qquad
v: = (H' +a)^{-1}g.
\end{equation}
Using the above formulae and the definition
of the operator $W$ we obtain
\[
\begin{split}
(Wf,g)_{L^2} &=
\big((H +a)^{-1}f,g\big)_{L^2}-
\big((H'+a)^{-1}f,g\big)_{L^2} \\
&= (u,g)_{L^2}\! -\! (f,(H' +a)^{-1}g)_{L^2}
=\big(u,(H' +a)v\big)_{L^2} -
\big((H +a)u,v\big)_{L^2}\\
&=
(u,H'v)_{L^2} - (Hu,v)_{L^2}.
\end{split}
\]
This formula can be rewritten in a more suitable way. Observe
that both functions $u$ and $v$ belong to
$H^1(\dR_+)$, which is the form domain of the operators
$H$ and $H'$. Hence,  using \eqref{fra} and
the first representation we get
\[
(Wf,g)_{L^2} = \fra_{X',\alpha'}[u,v]-\fra_{X,\alpha}[u,v]=(\alpha' \tau' u, \tau' v)_{\ell^2}-
(\alpha \tau u, \tau v)_{\ell^2},
\]
which can be further transformed into
\begin{equation}
\label{Wfg2}
(Wf,g)_{L^2} =
((\alpha' -\alpha) \tau u, \tau v)_{\ell^2}+
(\alpha' (\tau' -\tau) u, \tau v)_{\ell^2}
+(\alpha' \tau' u, (\tau' -\tau) v)_{\ell^2}.
\end{equation}
Define now the operators
\[
S_1 := \tau(H+a)^{-1},\quad S_2 := \tau'(H+a)^{-1},\quad \text{and}
\quad
S_3 := \tau(H'+a)^{-1}.
\]
Note that $(H+a)^{-1}$ and $(H'+a)^{-1}$ are bounded  as operators from
$L^2(\dR_+)$ into $H^1(\dR_+)$. Thus
we get that $S_k$ for $k=1,2,3$ are bounded as operators from $L^2(\dR_+)$
into $\ell^2(\dN)$.
Define another two operators,
\[
T_1 := (\tau' - \tau)(H+a)^{-1}\quad\text{and}\quad
T_2 := (\tau' - \tau)(H'+a)^{-1}.
\]
According to the result of Step 1
we obtain that $T_1$ and $T_2$
are compact from $L^2(\dR_+)$ into $\ell^2(\dN)$.
With the above notations in hands and using \eqref{uv}
we can rewrite \eqref{Wfg2} as
\[
(Wf,g)_{L^2} =
\big((\alpha' - \alpha)S_1f,S_3g)_{\ell^2}+
(\alpha'T_1f,S_3g)_{\ell^2} + (\alpha'S_2f,T_2g)_{\ell^2}.
\]
Hence, we get the following formula:
\[
W = S_3^*(\alpha' - \alpha)S_1 + S_3^*\alpha'T_1 + T_2^*\alpha'S_2.
\]
Boundedness of
$S_1, S_3$  and $\alpha_n - \alpha_n'\rightarrow 0$  imply compactness of $S_3^*(\alpha' - \alpha)S_1$.
Compactness of $T_1$, $T_2$ and boundedness of $S_2$, $S_3$
yield compactness of $S_3^*\alpha'T_1$ and $T_2^*\alpha'S_2$.
Thus the operator $W$ is compact.
Hence, by standard arguments which can be found for instance
in \cite[Lemma 6.21]{Te}, the resolvent difference
\eqref{resdiff} is then compact for all $\lambda\in\rho(H)\cap\rho(H')$.
\end{proof}

\begin{cor}
\label{cor:ess}
Let the discrete sets $\wt X$ and $\wh X$ be as in \eqref{wt X}
and \eqref{wh X}, respectively. Let the sequences $\wt \alpha$
and $\wh\alpha$ be, respectively, as in \eqref{wt alpha} and
\eqref{wh alpha}.  Let the self-adjoint operators
$H_{\varkappa,\wt X,\wt\alpha}$ and $H_{\varkappa,\wh X,\wh\alpha}$
be associated with $\wt X,\wt\alpha$
and $\wh X,\wh\alpha$, respectively, as in
Section~\ref{sec:def}.
Then  the relation
\[
\sess(H_{\varkappa,\wt X,\wt \alpha}) =
\sess(H_{\varkappa,\wh X,\wh \alpha}) =
\big\{ \lambda\in{ \dR}\colon
{ L}(\sqrt{\lambda}) \in [-1,1]\big\}
\]
holds, where ${ L}(\cdot)$ is as in
\eqref{Lyapunov function for delta}. In particular,
spectra of $H_{\varkappa,\wt X,\wt \alpha}$ and
$H_{\varkappa,\wh X,\wh\alpha}$ are discrete in the set
$\dR\setminus \big\{ \lambda\in { \dR}\colon
{ L}(\sqrt{\lambda}) \in [-1,1]\big\}$.
\end{cor}
\begin{proof}
Recall that the operator $H_{\varkappa}$ corresponding
as in Section~\ref{sec:def}
to the discrete set $\{nd\colon n\in\dN\}$ and constant
interaction strength $\alpha_0{\in\dR}$ has the essential spectrum
$
\{\lambda\in { \dR}\colon
{ L}(\sqrt{\lambda}) \in [-1,1]\big\}$.
The claims follow from $\wt x_n - nd = 0$, $\wt\alpha_n -\alpha_0 \rightarrow 0$,
$\wh x_n - nd \rightarrow 0$, $\wh\alpha_n -\alpha_0 = 0$
and Proposition~\ref{prop:compact}.
\end{proof}

\begin{rem}
Rybkin proved in \cite{R05} that for the half-line Schr\"odinger operator with distributional potential of the form $p  + q'$
with arbitrary $p,q\in L^2(\dR_+)$ the absolutely continuous
spectrum coincides with the interval $[0,+\infty)$.
It is not difficult to see that
$\sum_{n\in\dN}\alpha_n\delta_{x_n} = p + q'$
for
\[
p(x) = \sum_{n\in\dN}\tfrac{\alpha_n}{|x_{n}-x_{n-1}|}
\chi_{[x_{n-1},x_n]}(x),
\quad q(x) = -\sum_{n\in\dN}
\tfrac{\alpha_n(x-x_{n-1})}{|x_{n}-x_{n-1}|}
\chi_{[x_{n-1},x_n]}(x),
\]
where $\chi_I$ is the characteristic function
of the interval $I$. One may check that
$p,q$ as above are square-integrable for $X =\wt X$
and $\alpha = \wt\alpha$ with $\alpha_0 = 0$.
Note that Corollary~\ref{cor:main}\,(i) and Corollary~\ref{cor:ess} imply in the special case $\alpha_0 = 0$  that $\sigma_{\rm ac}(H_{\varkappa,\wt X,\wt\alpha}) = [0,+\infty)$,
which agrees well with the result of \cite{R05}.
\end{rem}

%*************************************************************
\section*{Appendix}
%*************************************************************
Recall that the transfer matrix $T_n(k)$
corresponding to the spectral parameter
 $\lambda  \in\dR\setminus\{0\}$ with $k := \sqrt{\lambda}$
for the operator with point interactions
supported on the discrete set
$X = \{x_n\colon n\in\dN\}\subset\dR_+$
with coupling constants $\alpha = \{\alpha_n\}_{n\in\dN}$
has the form
\begin{equation}\label{T-n-k appendix}
    T_n(k)=\left(
    \begin{array}{cc}
    0 & 1 \\
    -\frac{s_n(k)}{s_{n-1}(k)}
	&
    \frac{\sin(k(x_{n+1}-x_{n-1}))}{s_{n-1}(k)}+\frac{\alpha_ns_n(k)}{k}   \\
    \end{array}
    \right),
\end{equation}
 where
$s_n(k) = \sin(k(x_{n+1} -x_n))$ and
$c_n(k) = \cos(k(x_{n+1}-x_n))$, and it is implicitly
assumed that $s_n(k)\ne 0$ for all sufficiently
large $n\in\dN$. For $k=0$ we use the formula
\begin{equation*}
T_n(0) :=\left(
    \begin{array}{cc}
    0 & 1 \\
    -\frac{x_{n+1}-x_{n}}{x_n -x_{n-1}}
	&
    \frac{x_{n+1}-x_{n-1}}{x_n -x_{n-1}}+\alpha_n(x_{n+1} -x_n)  \\
    \end{array}
    \right)
\end{equation*}
instead of~\eqref{T-n-k appendix}. Our aim in  this appendix is { to reduce}
the transfer matrix $T_n(k)$ subject to Models~I and~II
to the form of \eqref{system on u} with subsequent computation
of the corresponding parameters $z_{\pm}$ defined in \eqref{betaphi}.
%-------------------------------------------------------------
\subsection*{Model I}
%-------------------------------------------------------------
In this section we deal with the discrete set
$ \wt X$ as in \eqref{wt X}
and the sequence of interaction
strengths $\wt\alpha$ as in \eqref{wt alpha}.
Assume in what follows that $k\notin{ \tfrac{\pi\dN_0}{d}}$
in which case $s_n(k)\ne 0$ for all $n\in\dN$.
The transfer matrix $T_n(k)$ subject to $X  = \wt X$ and
$\alpha = \wt\alpha$ can be rewritten in the form
\begin{equation}
\label{T-appendix-new}
\begin{split}
T_n(k) &=
\left(
    \begin{array}{cc}
    0 & 1 \\
    -1
	&
    \frac{\sin(2kd)}{\sin(kd)}+\frac{\wt \alpha_n\sin(kd)}{k}   \\
    \end{array}
    \right)\\
    & =
    \left(
    \begin{array}{cc}
    0 & 1 \\
    -1
	&
    2l   \\
    \end{array}
    \right) +
	\left(
    \begin{array}{cc}
    0 & 0 \\
    0	& \frac{c\sin(kd)}{k}\\
    \end{array}
    \right)\frac{\sin(2\omega n+\phi)}{n^\gamma}
    + \left(
    \begin{array}{cc}
    0 & 0 \\
    0	& \frac{q_n\sin(kd)}{k}\\
    \end{array}
    \right)
    \end{split}
\end{equation}
with $l = \cos(kd) + \alpha_0\tfrac{\sin(kd)}{2k}$.
Clearly, we have
\[
R_n :=
	\left(
    \begin{array}{cc}
    0 & 0 \\
    0	& \frac{q_n\sin(kd)}{k}\\
    \end{array}
    \right)\in\ell^1(\dN;\dC^{2\times 2}).
\]
Using the identity $\sin(2\omega n) =
\frac{1}{2i}(e^{2i\omega n} - e^{-2i\omega n})$
we get
\[
T_n(k)\! =\! \left(
    \begin{array}{cc}
    0 & 1 \\
    -1
	&
    2l   \\
    \end{array}
    \right) +
    \left(
    \begin{array}{cc}
    0 & 0 \\
    a	& b\\
    \end{array}
    \right)\frac{e^{2i\omega n}}{n^\gamma}
    + \left(
    \begin{array}{cc}
    0 & 0 \\
    \ov{a}	& \ov {b}\\
    \end{array}
    \right)\frac{e^{-2i\omega n}}{n^\gamma} + R_n,
\]
with
\[
a= 0\quad\text{and}\quad b = \frac{c\sin(kd)}{2ik}e^{i\phi}.
\]
Thus, according to \eqref{betaphi}  we arrive
at
\[
z_+ = -\frac{c\sin(kd)e^{i(\phi-2\omega)}}{4k\sin\omega}=-z_-.
\]
%Hence
%\[
%\beta = \beta_1 = \Bigg|\frac{c\sin(kd)}{4k\sin\omega}\Bigg|,
%\]
%and
%\[
%\varphi = -2\omega + \arg(-c\sin(kd)),\qquad
%\varphi_1 = -2\omega + \arg(c\sin(kd)).
%\]

%-------------------------------------------------------------
\subsection*{Model II}
In this section we deal with the discrete set
$ \wh X$ as in \eqref{wh X}
and the sequence of interaction
strengths $\wh\alpha$ as in \eqref{wh alpha}.
Assume in what follows that $k\notin{ \tfrac{\pi\dN_0}{d}}$
in which case $s_n(k)\ne 0$ for all sufficiently large $n\in\dN$.
We denote by $\{r_n^{(m)}\}_{n\in\dN}$ with $m\in\dN$  generic sequences from $\ell^1(\dN)$.
Note that the Taylor-type expansion
\begin{equation}
\label{snkexp}
\begin{split}
s_n(k)
&=
\sin\Big(kd + ck\Big(\tfrac{\sin(2\omega(n+1)+\phi)}{(n+1)^\gamma}-
\tfrac{\sin(2\omega n+\phi)}{n^\gamma}\Big) + k(q_{n+1} -q_n)\Big) \\
&= \sin(kd) +
\tfrac{ck\cos(kd)\big(\sin(2\omega(n+1)+\phi) -\sin(2\omega n+\phi)\big)}{n^\gamma} + r_n^{(1)}\\
\end{split}
\end{equation}
holds. Using the above expansion we obtain
\[
\begin{split}
\tfrac{s_n(k)}{s_{n-1}(k)} &=
\tfrac{1 + ck\cot(kd)n^{-\gamma}
\big(\sin(2\omega(n+1)+\phi) -\sin(2\omega n+\phi)\big) + r_n^{(1)}}
{1 + ck\cot(kd)n^{-\gamma}
\big(\sin(2\omega n+\phi) -\sin(2\omega (n-1)+\phi)\big) + r_{n-1}^{(1)}} \\
&=
1 + \tfrac{ck\cot(kd)\big(\sin(2\omega(n+1)+\phi) + \sin(2\omega(n-1)+\phi)
- 2\sin(2\omega n+\phi)\big)}{n^\gamma}+r_n^{(2)}, \\
\end{split}
\]
which can be simplified making use of standard trigonometric identities to
\begin{equation}
\label{21}
\tfrac{s_n(k)}{s_{n-1}(k)} = 1 -
\tfrac{4ck\cot(kd)\sin^2(\omega)\sin(2\omega n+\phi)}{n^\gamma}
+r_n^{(2)}.
\end{equation}
Again utilising the expansion \eqref{snkexp} we get
\begin{equation}
\label{22}
\tfrac{\alpha_0 s_n(k)}{k} =
\tfrac{\alpha_0\sin(kd)}{k} +
\tfrac{2\alpha_0c\cos(kd)\sin(\omega)\cos((2n+1)\omega+\phi)}{n^\gamma}
+ r_n^{(3)}.
\end{equation}
Using \eqref{snkexp} and the Taylor-type expansion
\[
\begin{split}
&\sin(k(\wh x_{n+1}\! -\! \wh x_{n-1})) \\
&\qquad=\sin\Big(2kd + ck
\Big(\tfrac{\sin(2\omega(n+1)+\phi)}{(n+1)^\gamma}-
\tfrac{\sin(2\omega (n-1)+\phi)}{(n-1)^\gamma}\Big)
+ k(q_{n+1} - q_{n-1})\Big) \\
&\qquad= \sin(2kd) +
\tfrac{ck\cos(2kd)\big(\sin(2\omega(n+1)+\phi) - \sin(2\omega (n-1)+\phi)\big)}{n^\gamma}+ r_n^{(4)}\\
\end{split}
\]
we arrive at
\[
\begin{split}
&\tfrac{\sin(k(\wh x_{n+1} - \wh x_n))}{s_{n-1}(k)} =
\tfrac{2\cos(kd) +
ck\cos(2kd)(\sin(kd)n^\gamma)^{-1}
\big(\sin(2\omega(n+1)+\phi) - \sin(2\omega (n-1)+\phi)\big)
+ r_n^{(4)}}
{1 + ck\cot(kd)n^{-\gamma}
 \big(\sin(2\omega n+\phi) - \sin(2\omega (n-1)+\phi)\big)+r_{n-1}^{(1)}}\\
&\quad= 2\cos(kd) +
\tfrac{2ck\cos(2kd)
\sin(2\omega)\cos(2\omega n+\phi)}{\sin(kd)n^\gamma}
-\tfrac{4ck\cos^2(kd)
\sin\omega\cos((2n-1)\omega+\phi)}{\sin(kd)n^\gamma}
 +r_n^{(5)},\\
\end{split}
\]
which can be further rewritten as
\begin{equation}
\label{22'}
\begin{split}
\tfrac{\sin(k(\wh x_{n+1} - \wh x_n))}{s_{n-1}(k)}&\!=\!2\cos(kd) +
w_n + r_n^{(5)}, \quad\text{where}\\
w_n & := \tfrac{
4ck\cos(kd)\sin\omega\big(
2\cot(2kd)\cos\omega\cos(2\omega n+\phi)
- \cot(kd)\cos((2n-1)\omega+\phi)\big)}{n^\gamma}.
\end{split}
\end{equation}
Employing \eqref{21}, \eqref{22}, the above formulae,
and the identities
\[
\sin\alpha=
\tfrac{e^{i\alpha} - e^{-i\alpha}}{2i},\qquad
\cos\alpha =
\tfrac{e^{i\alpha} + e^{-i\alpha}}{2},
\]
the transfer matrix $T_n(k)$ subject to $X  = \wh X$ and
$\alpha = \wh\alpha$ can be rewritten in the form
\[
T_n(k) =
\begin{pmatrix}
0 & 1\\
-1 & 2l
\end{pmatrix} +
\begin{pmatrix}
0 &  0\\
a & b
\end{pmatrix}\frac{e^{2i\omega n}}{n^\gamma}
+\begin{pmatrix}
0 &  0\\
\ov{a} & \ov{b}
\end{pmatrix}\frac{e^{-2i\omega n}}{n^\gamma} +R_n
\]
with
\[
\begin{split}
 l  = & \cos(kd) + \frac{\alpha_0 \sin(kd)}{2k},\\[0.5ex]
 a = & -2ikc\cot(kd) \sin^2(\omega)e^{i\phi},\\[0.5ex]
 b = & c\cos(kd)\sin(\omega)\Big[4k\cos(\omega)\cot(2kd) - 2k\cot(kd)e^{-i\omega} +\alpha_0e^{i\omega}\Big]e^{i\phi}.
\end{split}
\]
To compute the value
\[
z_+=\frac{a e^{-i\omega} + be^{-2i\omega}}{2i\sin\omega}
\]
for $l=\cos\omega$ we define an auxiliary value
\begin{equation*}
\label{z'}
z_+' :=z_+ e^{i (\omega-\phi) }2i\sin\omega  =
e^{-i\phi}(a  + be^{-i\omega}).
\end{equation*}
For the real part of $z_+'$
we get
\[
\Re(z_+')  =
c\cos(kd)\sin\omega
\big[
4k\cos^2\omega\cot(2kd) - 2k\cot(kd)\cos(2\omega) + \alpha_0\big].
\]
Simplification
\[
\begin{split}
&4k\cos^2\omega\cot(2kd) - 2k\cot(kd)\cos(2\omega)
= \tfrac{4k\big(\cos^2(\omega)\cos(2kd)-\cos^2(kd)\cos(2\omega)\big)}{\sin(2kd)}
\\
&\qquad= \tfrac{4k\big(\cos^2(kd)-\cos^2(\omega)\big)}{\sin(2kd)}=
-\tfrac{\alpha_0\sin(kd)}{2k}\tfrac{4k\big(\cos(kd)+\cos(\omega)\big)}{\sin(2kd)}  =
 -\alpha_0\left(1+\tfrac{\cos(\omega)}{\cos(kd)}\right),
\end{split}
\]
where we have used the equality $l=\cos(\omega)=\cos(kd)+\frac{\alpha_0\sin(kd)}{2k}$,
gives us
\begin{equation*}
\label{Re}
\Re(z_+') =-\alpha_0c\sin\omega\cos\omega.
\end{equation*}
For the imaginary part of $z_+'$ we get
\[
\begin{split}
\Im(z_+')  & =
-2kc\cot(kd)\sin^2\omega +
c\cos(kd)\sin\omega\sin(2\omega)
\big[-2k\cot(2kd) + 2k\cot(kd)\big] \\
&=
\tfrac{2kc\sin^2(\omega)\big(\cos(\omega)-\cos(kd)\big)}{\sin(kd)}
=\alpha_0c\sin^2(\omega),
\end{split}
\]
which gives us
\[
z_+'=-\alpha_0c\sin\omega e^{-i\omega}
\]
and
\[
z_+=\frac{ic\alpha_0e^{i(\phi-2\omega)}}2.
\]
To compute $z_-$
\[
z_- = \frac{ae^{-i\omega} - be^{-2i\omega}}{2i\sin\omega}
\]
for $l = -\cos\omega=\cos(kd)+\frac{\alpha_{0}\sin(kd)}{2k}$ we define
\begin{equation*}
\label{z-'}
z_-':=  z_-e^{i(\omega-\phi)}2i\sin\omega =e^{-i\phi}( a-be^{-i\omega}).
\end{equation*}
Similarly as above we obtain
\begin{equation*}
\label{ReIm'}
\begin{split}
\Re(z_-') &= -
\tfrac{ck\sin\omega\big[2(\cos^2(kd) - \cos^2(\omega)) + \tfrac{\alpha_0}{2k}\sin(2kd)\big]}{\sin(kd)}=-\alpha_0c\sin\omega\cos\omega,\\
\Im(z_-') &=
\tfrac{ck\sin\omega\big[-2\cos(kd)\sin\omega - \sin(2\omega)\big]}{\sin(kd)}=\alpha_0c\sin^2\omega.
\end{split}
\end{equation*}
Thus $z_-'=z_+'$ and hence $z_-=z_+$.

In the case $k=0$ one has to replace the expressions { in} all the formulas by their values at zero or { their} limits as $k\rightarrow0$. In particular, the expression~\eqref{T-appendix-new} takes the form
\[
\begin{split}
T_n(0) &=
    \left(
    \begin{array}{cc}
    0 & 1 \\
    -1
	&
    2l   \\
    \end{array}
    \right) +
	\left(
    \begin{array}{cc}
    0 & 0 \\
    0	& cd\\
    \end{array}
    \right)\frac{\sin(2\omega n+\phi)}{n^\gamma}
    + \left(
    \begin{array}{cc}
    0 & 0 \\
    0	& q_nd\\
    \end{array}
    \right)
    \end{split},
\]
in~\eqref{21} the right-hand side should be replaced by
\[
1-\frac{4c\sin^2(\omega)\sin(2\omega n+\phi)}{dn^{\gamma}}+r_n^{(2)},
\]
in~\eqref{22} by
\[
\alpha_0d+\frac{2\alpha_0c\sin(\omega)\cos((2n+1)\omega+\phi)}{n^{\gamma}}+r_n^{(3)},
\]
and in~\eqref{22'} by
\[
2+\frac{4c\sin\omega(\cos\omega { \cos}(2\omega n+\phi)-\cos((2n-1)\omega+\phi))}{dn^{\gamma}}+r_n^{(5)}.
\]
}

%-------------------------------------------------------------
\subsection*{Acknowledgments}
%-------------------------------------------------------------

The authors wish to express their gratitude to
Prof. Sergey Naboko
for his constant attention to this work
and many fruitful discussions on the subject.
Prof. G\"unter Stolz and Dr. Aleksey Kostenko
are acknowledged for important comments and valuable help.
V.~L. thanks Technische Universit\"at Wien for hospitality,
where a part of this work was done. His work
was supported by the Austrian  Science Fund (FWF), project P 25162-N26.
S. S. was supported by the Chebyshev Laboratory
 (Department of Mathematics and Mechanics, Saint-Petersburg State University) under the grant 11.G34.31.0026 of the Government of the Russian Federation, by grants  RFBR-09-01-00515-a, RFBR-12-01-00215-a and 11-01-90402-Ukr\_f\_a, by the Erasmus Mundus Action 2 Programme of the European Union, and Irish Research Council (Government of Ireland Postdoctoral Fellowship in Science, Engineering and Technology).

\end{document}